\documentclass[12pt,x11names]{article}
\usepackage{latexsym}
\usepackage[dvips]{graphicx}
\usepackage{amsfonts,amssymb,amsmath}
\usepackage{pdfsync}
\usepackage{epsfig}
\usepackage{geometry}
\usepackage{color}
\usepackage{pstricks}
\usepackage{setspace}
\usepackage{comment}
\usepackage{enumerate}
\usepackage{hyperref}
\usepackage{wrapfig}
\usepackage{pgfplots}
\usepackage{tikz}
\usepackage{float}
\usepackage{cases}
\usetikzlibrary{datavisualization}
\usetikzlibrary{datavisualization.formats.functions}
\newtheorem{theorem}{Theorem}
\newtheorem{proposition}[theorem]{Proposition}
\newtheorem{lemma}[theorem]{Lemma}
\newtheorem{corollary}[theorem]{Corollary}

\newtheorem{definition}{Definition}

\newtheorem{ejem}{Example}%[section]

\newcommand{\N}{\mathbb{N}}

\newcommand{\R}{\mathbb{R}}

\newcommand{\KK}{\mathcal{K}}

\newcommand{\Rcupinf}{\R\cup\{+\infty\}}

\newcommand{\f}{\varphi}

\newenvironment{proof}{\paragraph{Proof.}}{\hfill$\square$}

\DeclareMathOperator{\dist}{dist}
\DeclareMathOperator{\cl}{cl}
\DeclareMathOperator{\bd}{bd}

\DeclareMathOperator{\interior}{int}

\DeclareMathOperator{\dom}{dom}

\title{A stroll in the jungle of error bounds}
\author{Trong Phong Nguyen\footnote{TSE (Universit\'{e} Toulouse I Capitole), Manufacture des Tabacs, 21 all\'{e}e de Brienne, 31015 Toulouse, Cedex 06, France. E-mail: trong-phong.nguyen@ut-capitole.fr. The author thanks the Air Force Office of Scientific Research, Air Force Material Command, USAF, under grant number FA9550-14 -1-0056 and FA9550-14-1-0500 for their partial support.}}

\begin{document}
%\chapter{Error Bounds}
	\maketitle
	% \pdfoutput=1
%\paragraph{Abstract}	
\begin{abstract}
The aim of this paper is to give a short overview on error bounds and to provide the first bricks of a unified theory. Inspired by the works of \cite{AzeCor02,BolDanLeyMaz, BolDanLew1, eb&kl,AzeCor16}, we show indeed the centrality  of the \L ojasiewicz gradient inequality. For this, we review some necessary and sufficient conditions for global/local error bounds, both in the convex and nonconvex case. We also recall some results on quantitative error bounds which play a major role in convergence rate analysis and complexity theory of many optimization methods.
\end{abstract}

	\section{Introduction}
Let $X$ be a Banach space. Given a  function $f\colon X\rightarrow \Rcupinf$, an error bound is an inequality that bounds the distance from an arbitrary point  in a test set to the level set in terms of the function values. More precisely, we shall say that $f$ has an error bound on a set $K\subset X$ if there exists an increasing function $\varphi\colon [0, +\infty) \rightarrow [0,+\infty), \varphi(0)=0$ such that 
\begin{equation}\label{EB}	\dist(x,[f\leq 0])\leq \varphi(f(x)),\, \forall x\in K.\end{equation}
When $K=X$ then $f$ is said to possess global error bound; otherwise we say that $f$ has a local error bound.

Error bounds have a lot of applications in many fields. They may be used to establish the rate of convergence of many optimization methods: we can think to descent methods for solving minimization problems \cite{amirbeck, DruLew, LuoTseng, LuoTse922, LuoTseng93, TseYun}, to the cyclic projection algorithm \cite{BeckTeboulle, Borwein,LMNP}, to algorithms for solving  variational inequalities, see e.g., \cite{Tse}.  In \cite{eb&kl, WangLin} the error bound theory is used to estimate the complexity of a wealth of descent methods for convex problems. Error bounds have also played a major role in the context of metric regularity \cite{Aus, Iof79, Kru15} or within the field of exact penalty functions, see e.g., \cite{Dedieu00}.

\smallskip

Let us now present the two major mathematical results that are structuring the theory of error bounds and along which we will develop our own presentation. Hoffman seems to be the first to provide an error bound in the context of optimization theory. His result concerns affine function system:
	\begin{theorem}[Hoffman, 1952]\cite{Hof} 
Let $A, B\in\R^{m\times n}$ be some matrices and $a,b$ are the vectors in $\R^m$. Assume that 
		$$S=\left\lbrace x\in \R^n|Ax\leq a, Bx=b\right\rbrace.$$
is nonempty. Then, there exists a scalar $c>0$ such that
		$$\dist(x,S)\leq c\left(\|[Ax-a]_+\|+\|Bx-b\|\right), \forall x\in \R^n.$$
	\end{theorem}

Around the same time, a very general and powerful result was provided in \cite{Loja59} for semi-algebraic functions. This result was developed as a positive response to a conjecture of Schwartz in the distribution theory of functions (see \cite{LuoPang}). Later, in \cite{Hiro}, Hironaka extended this inequality to the case of subanalytic function.
	\begin{theorem}[\L ojasiewicz, 1959]\label{loja59}\cite{Hiro, Loja59}
Let $\phi,\, \psi\colon \R^n\longrightarrow \R $ be two continuous subanalytic functions. If $\phi^{-1}(0)\subset \psi^{-1} (0)$ then for each compact, subanalytic set $K\subset \R^n$, there exist a constant $c>0$ and a integer $N$ such that
		$$ c |\psi (x)|^N \leq |\phi(x)|, \, \forall x\in K.$$	\end{theorem}

After those pioneering works, the study of error bounds has attracted numerous researches. In 1972, under a Slater's condition and a boundedness assumption on the level sets, Robinson \cite{Rob}  extended the result of Hoffman to systems of convex differentiable inequalities. Mangasarian \cite{Man} established the same result for the maximum of finitely many differentiable convex functions. Later on, Auslender and Crouzeix \cite{Aus}, extended Mangasarian's result to non-differentiable convex functions. Some other sufficient conditions were also given by Deng in \cite{Deng97, Deng98}, by using in particular a Slater's condition on the recession function. In \cite{LewPang}, Lewis and Pang gave a characterization of Lipschitz global error bound for convex functions in terms of the directional derivatives. The work of Lewis and Pang was further generalized by Ng and Yang \cite{NgFeng1}, by Wu and Yu in \cite{WuYe01, WuYe02}, by Klatte and Li in \cite{Klatte}. In a series papers \cite{Aze, AzeCor02, AzeCor2,AzeCor14, AzeCor16}, Az\'e and Corvellec presented some characterizations of  error bounds in terms of the strong slope in the context of metric spaces. 

The first fundamental works on quantitative error bounds seem to be those of Gwozdziewicz \cite{Gwo} and Koll\'ar \cite{Kollar} for polynomial functions. Inspired by these works many researchers have tried to provide more general types of quantitative error bounds. Li, Murdokhovich and Pham \cite{LiMorPham} established a local error bound for polynomial function systems in the nonconvex case. Li \cite{Li10}, and Yang \cite{Yang} obtained some error bounds for polynomial convex functions, the work of Li was extended for piecewise convex polynomial function in \cite{Li}, which has also improved the result of Li \cite{WLi95}. In \cite{Ngai}, Ngai gave some similar results on polynomial function systems. For the quadratic function systems, Luo and Luo \cite{LuoLuo} seem to be the first to have studied global error bounds for such class, under the assumption of convexity. This work has been improved by Pang and Wang \cite{Pangwang} and later by Luo and Sturm \cite{LuoSturm} who derived a 
 global error bound for such function without assuming convexity. 
 
The connection between error bound and Kurdyka--\L ojasiewicz inequality was first settled by Bolte, Daniilidis, Ley and Mazet, in \cite{BolDanLeyMaz}. Later some of these results were improved  \cite{eb&kl, AzeCor16}. 

\medskip

This paper is organized as follows:

In Section~3, based on the results in \cite{AzeCor02,BolDanLeyMaz,AzeCor16}, we give a characterization of error bounds, specifying this result in some particular cases. We establish the connection between this result and some other previous sufficient conditions for Lipschitz error bounds.  

In Section 4, we review some results on local error bounds and  global error bound, respectively. We focus on the class of polynomial functions whose error bounds are of H\"older type, which play a major role in complexity theory of many optimization methods.
 
	\section{Preliminaries}
Let $X$ be a Banach space, $X^*$ be topological dual space and $f \colon X \longrightarrow \R$ be a lower semicontinuous function. For any $\alpha,\, \beta\in \mathbb R$, we set $[f\leq  \alpha]=\left\lbrace x\in X| f(x)\leq  \alpha\right\rbrace$, $[f=\alpha]=\left\lbrace x\in X| f(x)= \alpha\right\rbrace$, $[\alpha\leq f\leq \beta]=\left\lbrace x\in X| \alpha\leq f(x)\leq \beta \right\rbrace$, and $[\alpha]_+=\max \{\alpha,0\}$. For any subset $S\subset X$, denote $\dist(x,S)=\inf_{u\in S} \|x-u\|$, and $\bd S, \, \cl S,\, \interior S$ respectively are the boundary, closure and interior set of $S$. For $x\in X, \,\delta>0$, set $B_\delta(x)=\{y\in X|\dist(x,y)<\delta \}$.

The Fr\'echet subdifferential of $f$ at $x\in \dom f$, denote $\partial^F f(x)$, is defined by
 $$\partial^F f(x)=\left\lbrace u\in X^*| \liminf_{y\to x}\frac{f(y)-f(x)-\langle u,y-x\rangle}{\|y-x\|}\geq 0\right\rbrace.$$
 The limiting-subdifferential of $f$ at $x\in \dom f$, written $\partial f(x)$ is defined as follows
 $$\partial f(x)=\left\lbrace u\in X^*|\exists x_k\rightarrow x,\, f(x_k)\rightarrow f(x),\, u_k\in \partial^F f(x_k)\rightarrow u\right\rbrace.$$
And
$$f'(x,d)=\liminf_{t\to 0^+}\frac{f(x+td)-f(x)}{t},$$
is called the derivative of $f$ at $x$ in the direction $d\in X$.

The strong slope of $f$ at $x$ is given by
$$|\nabla f|(x)=\begin{cases}
0&\text{ if $x$ is a local minimum point of $f$,}\notag\\
\limsup_{y\to x}\frac{f(x)-f(y)}{\|x-y\|}& \text{ otherwise}.\notag
\end{cases}$$
It is easy to see that
\begin{itemize}
	\item $\|d\| |\nabla f|(x)\geq -f'(x,d),\:\: \forall (x,d)\in X^2$.
	\item $|\nabla f| (x)\leq \dist(0,\partial^F f(x)), \:\: \forall x\in X.$
%	\item $\displaystyle |\nabla f| (x)\geq \dist(0,\partial f(x)), \forall x\in X$, see \cite[Remark 6.1]{AzeCor14}.
\end{itemize}
We recall the chain rule for the strong slope.
\begin{lemma}\cite{AzeCor16}\label{chain}
Let $-\infty<\alpha<\beta \leq +\infty $ and a function $\f\colon ]\alpha, \beta [\rightarrow \R$ with $\f\in C^1(\alpha,\beta)$ and $\f'(s)>0,\, \forall s\in ]\alpha,\beta[$. One has
$$|\nabla (\f\circ f)|(x)=\f'(f(x))|\nabla f|(x), \, \forall x\in [\alpha <f<\beta].$$	
\end{lemma}
As mentioned in the introduction, the theory of error bound can be developed, based on the theory of \L ojasiewicz on the subanalytic function. Let us recall the definition of such function class. 
\begin{definition}\cite{BolDanLew1}
\begin{description}
	\item[(i)] A function $f\colon\R^n \rightarrow \R$ is real-analytic on $S\subset \R^n$ if it can be represented locally on $S$ by a convergent power series, this means that, for any $\bar{x}=(\bar{x}_1,\ldots,\bar{x}_n)\in S$, there exists a neighborhood $U(\bar{x})$ such that 
	$$f(x)=\sum_{i_1,\ldots,i_n=0}^\infty a_{i_1,\ldots,i_n}(x_1-\bar{x}_1)^{i_1}\ldots (x_n-\bar{x}_n)^{i_n}, \,\forall x\in U(\bar{x}).$$ 
	\item[(ii)]
	A subset $S$ of $\R^n$ is called semianalytic if for each point $\bar{x}\in \R^n$ admits a neighborhood $U(\bar{x})$ for which $S\cap U(\bar{x})$ is represented in the following form
	$$S\cap U(\bar{x})=\bigcup_{i=1}^p\bigcap_{j=1}^q \left\lbrace x\in \R^n|f_{ij}(x)=0,\, g_{ij}(x)>0\right\rbrace,$$
where the functions $f_{ij},\,g_{ij}\colon\R^n\rightarrow \R$ are real-analytic for all $1\leq i \leq p,\, 1\leq j \leq q$. If the graph of $f$ is a semianalytic set in $\R^{n+1}$, we say that $f$ is a semianalytic function.

\item[(iii)] $S$ is called subanalytic if each point $\bar{x}\in \R^n$, there exist a neighborhood $U(\bar x)$ and a bounded semianalytic set $A\subset \R^{n+m}$, (for some $m\in \N^*$) such that $S\cap U(\bar{x})$ is the projection on $\R^n$ of $A$. The function $f$ is subanalytic if its graph is a subanalytic set in $\R^{n+1}$.
\end{description}	
\end{definition}
We give some elementary properties of subanalytic function and subanalytic set, see\cite{BolDanLew1}.
\begin{enumerate}
	\item If $S$ is subanalytic set then so are its boundary $\bd S$, its closure $\cl S$, its interior $\interior S$, and its complement set.
	\item The class of subanalytic sets is closed under finite union and intersection. The distance function to a subanalytic set is a subanalytic function.
	\item The image of a bounded subanalytic set under a subanalytic map is
	subanalytic. The inverse image of a subanalytic set under a subanalytic map is a subanalytic set.
	\item When $S$ is a closed, convex subanalytic set, the Euclidean
	projector onto $S$ is a subanalytic function.
\end{enumerate}
In this work, we focus on the H\"older-type error bound, which is very common in practice. 
\begin{definition}
Let $f$ be a function on the Banach space $X$ and $K$ be a subset of $X$. We say that $f$ admits a
\begin{enumerate}
	\item H\"older-type error bound on $K$ if there exists $\tau>0$ and $ a, b>0$ such that 
	\begin{equation}\label{Hoder}
	\dist(x,[f\leq 0])\leq  \tau\left( [f(x)]_+^a+[f(x)]_+^b\right),\quad \forall x \in K. 
		\end{equation}
	\item Lipschitz-type (or linear) error bound on $K$ if the inequality (\ref{Hoder}) holds with $a=b=1$, for all $x\in K$.
\end{enumerate}	
When $K\equiv X$ then $f$ is said to have a global error bound, otherwise we say that $f$ possesses a local error bound.	
\end{definition}

\section{Characterization of error bounds}
\subsection{Characterizing error bounds through Kurdyka--\L ojasiewicz ine\-qua\-lity}

The \L ojasiewicz gradient inequality was introduced in \cite{Loja63}. Let $f\colon\R^n\rightarrow \R$ be an analytic function. For any $\bar{x}\in \dom f$, there exist $\tau>0,\,\theta \in [0,1)$ and a neighbourhood $U(\bar{x})$ such that
	$$\|\nabla f(x)\|\geq \tau \left|f(x)-f(\bar{x})\right|^\theta, \, \forall x\in U(\bar{x}).$$
In \cite{Kur98}, Kurdyka generalized the above result to the class of $C^1$ functions whose graphs belong to an o-minimal structure (the definition of o-minimal structure can be seen in \cite[Definition 1]{Kur98}, \cite[Definition 6]{BolDanLewShi07}), this result was extended to the nonsmooth class by Bolte,  Danillidis, Lewis and Shiota \cite{BolDanLewShi07}. The corresponding generalized \L ojasiewicz gradient inequality is called the Kurdyka--\L ojasiewicz (KL inequality for short) inequality. In addition, the generalization for the class nonsmooth subanalytic functions has been obtained by Bolte, Danillidis and Lewis in \cite{BolDanLew1}. This has opened the road to many theoretical and algorithmic developments (see \cite{AbsMahAnd,attbol,AttBolRedSou, AttBolSva, BST,eb&kl,PierreGuiJuan}). We summarize the above extension by the following theorem.
\begin{theorem}\cite{BolDanLewShi07, BolDanLew1}
 Let $f\colon \R^n\rightarrow \R$ be a definable function in an arbitrary o-minimal structure over $\R^n$. Then for all $\bar{x}\in\dom f$, there exist $\delta>0$ and a neighbourhood $U(\bar{x})$ of $\bar{x}$ such that 
$$\varphi'\left( f(x)-f(\bar{x})\right)\dist\left(0, \partial f(x)\right)\geq 1,\, \forall x\in U(\bar{x})\cap [f(\bar{x})<f<f(\bar{x})+\delta],$$
where $\varphi\colon[0,\delta]\rightarrow [0,+\infty)$ is an increasing function, which vanishes at zero and $\varphi\in C^0[0,\delta]\cap C^1(0,\delta)$. The class of such functions $\varphi$ will be denoted by $\KK[0,\delta]$.
\end{theorem}   
The connection between error bounds and Kurdyka--\L ojasiewicz inequality was established in \cite{BolDanLeyMaz} (see also \cite{eb&kl}), this was further improved by Az\'e and Corvellec \cite{AzeCor16}.

Az\'e and Corvellec have series of researches on the characterization of global error bounds for lower semicontinuous functions in terms of the strong slope, see \cite{AzeCor02, AzeCor2, AzeCor14, AzeCor16, corv}. These works are of great help for this section. Let us now give the result in \cite{AzeCor02}, in which, the authors used Ekeland's variational principle to establish the connection between the strong slope and the linear error bound.

\begin{theorem}\cite{AzeCor02}\label{AzeCorLinear}
Let $f\colon X\rightarrow \R\cup \{+\infty\}$ be a lower semicontinuous function, and $-\infty<\alpha<\beta\leq \infty$. Then
$$\inf_{x\in [\alpha<f<\beta]} |\nabla f|(x)=\inf_{\alpha\leq \gamma<\beta}\left(\inf_{x\in [\gamma<f<\beta]}\frac{f(x)-\gamma}{\dist(x,[f\leq \gamma])}\right).$$
\end{theorem}
We rewrite the latter theorem as a characterization of linear global error bound, which is a well known result since Ioffe's pioneering works \cite{Iof00}.
% {\color{red} verify and quote}
\begin{theorem}\label{Linear}\cite{AzeCor02}
	Let $\tau>0$, the following assertions are equivalent
	\begin{description}
		\item[(i)] $|\nabla f|(x)\geq \frac{1}{\tau},\, \forall x\in [\alpha<f<\beta].$
	\item[(ii)] $\tau \left(f(x)-\gamma\right)\geq \dist\left(x,[f\leq \gamma]\right), \, \forall \gamma \in [\alpha, \beta),\,x\in [\gamma<f<\beta].$
	\end{description}
\end{theorem}
For any $\f \in \KK [0,\beta-\alpha]$, thanks to Lemma~\ref{chain}, we can apply the latter result for the function  $x\mapsto\varphi\left(f(x)-\alpha\right)$, therefore we obtain a nonlinear version of Theorem~\ref{Linear}.% which is given as following.
 \begin{theorem}\label{AzeCor}
Assume that $\varphi\in \KK[0,\beta-\alpha)$. The following statements are equivalent
\begin{description}
	\item[(i)] $\varphi'(f(x)-\alpha)|\nabla f|(x)\geq 1,\,\forall x\in [\alpha<f<\beta].$
	\item[(ii)] $\varphi \left(f(x)-\alpha\right)\geq \f(\gamma -\alpha)+\dist\left(x,[f\leq \gamma]\right), \,\forall \gamma\in [\alpha,\beta], \, \forall x\in [\gamma<f<\beta].$
	\end{description}	
\end{theorem}
This is content of \cite[Corollary 4]{BolDanLeyMaz}, \cite[Theorem 4.2]{AzeCor16}. In the latter result, if we let $\gamma$ equal to $\alpha$ in the assertion (ii), then we immediately obtain as a consequence, a sufficient condition for nonlinear global error bound.
\begin{corollary}\label{nonlinear}
We suppose that 
$$\varphi'(f(x)-\alpha)|\nabla f|(x)\geq 1,\,\forall x\in [\alpha<f<\beta],$$
where $\varphi\in \KK[0,\beta-\alpha)$. Then
 $$\varphi \left(f(x)-\alpha\right)\geq \dist\left(x,[f\leq \alpha]\right), \, \forall x\in [\alpha<f<\beta].$$
\end{corollary}
Generally, the converse of this corollary is false, as shown in \cite[Remark 3]{Kur14} (when $f$ is a polynomial function) and  in \cite[Theorem 28]{eb&kl} (when $f$ is convex). However, in some particular cases, this converse may be hold, for example:
\begin{itemize}
	\item $f$ is an analytic function with an isolated zero, see \cite{Gwo}.
	
	\item  $f$ is a convex function and an additional assumption on $\f$, see \cite{eb&kl}, (we also show this result in Theorem~\ref{c1:equiv}).
\end{itemize}
 
Recalling $\|d\||\nabla f| (x)\geq -f'(x,d), \forall (x,d)\in X^2$, a consequence of Theorem~\ref{Linear} is:
\begin{corollary}\label{Nesscond}
	For any $\tau>0$, suppose that for each $x\in [\alpha<f<\beta]$, there exists a unit vector $d_x\in X$ such that 
	$$f'(x,d_x)\leq -\frac{1}{\tau}.$$
Then 
$$\tau \left(f(x)-\alpha\right)\geq \dist\left(x,[f(x)\leq\alpha]\right), \,\forall x\in [\alpha<f<\beta].$$
\end{corollary}
This is the content of \cite[Theorem 2.5]{NgFeng1}. A local version of Theorem~\ref{AzeCor} is given as follows
\begin{theorem}\cite{AzeCor16}\label{kllocal}
%Let $f:X\rightarrow \Rcupinf$ be lower semicontinuous, $-\infty<\alpha<\beta\leq \infty$ and $\varphi\in \KK[0,\beta-\alpha]$.
Consider the following statements
	\begin{description}
		\item[(i)] There exists $\varepsilon>0$ such that
		$$\varphi'(f(x)-\alpha)|\nabla f|(x)\geq 1,\, \forall x\in B_\varepsilon(\bar{x}) \cap[\alpha<f<\beta].$$
		\item[(ii)] There exists $\rho>0 $ such that $$\varphi \left(f(x)-\alpha\right)\geq \f(\gamma -\alpha)+ \dist\left(x,[f(x)\leq\gamma]\right), \, \forall \gamma\in [\alpha,\beta),\,\forall x\in B_\rho(\bar{x}) \cap [\alpha<f<\beta].$$
	\end{description}		
	Then $(i)\Rightarrow (ii)$ with $\rho=\varepsilon/2$ and $(ii)\Rightarrow (i)$ with $\varepsilon=\rho$.
\end{theorem}
In the statement (ii), by setting $\gamma=\alpha$, we obtain a local version of Corollary~\ref{nonlinear}.
\begin{corollary}\cite{AzeCor16}\label{localnonlinear}
	%Let $f:X\rightarrow \Rcupinf$ be lower semicontinuous, $-\infty<\alpha<\beta\leq \infty$ and $\varphi\in \KK[0,\beta-\alpha]$.
	For any $\bar{x}\in [f\leq \alpha]$, suppose that there exists $\varepsilon>0$ such that
		$$\varphi'(f(x)-\alpha)|\nabla f|(x)\geq 1,\, \forall x\in B_{2\varepsilon}(\bar{x}) \cap[\alpha<f<\beta].$$
Then
 $$\varphi \left(f(x)-\alpha\right)\geq  \dist\left(x,[f(x)\leq\alpha]\right), \,\forall x\in B_\varepsilon(\bar{x}) \cap [\alpha<f<\beta].$$
	
\end{corollary}
If we take $\varphi (s)=\tau s^\theta,\,\tau>0,\,\theta \in [0,1]$, then this corollary recover the result of Ngai, Thera \cite[Corollary 2]{NgaiThera09}.

As we mentioned before, the converse of the latter corollary does not always hold. The results of Corollary~\ref{nonlinear}, Corollary~\ref{localnonlinear} have been appeared in numerous works, for instance, see \cite{Gwo,Kur14, NgaiThera05, NgaiThera09, Son10,Son12}. This result gives an useful tools for establishing the quantitative error bounds, see \cite{LiMorPham,LMNP}.   
\subsection{Equivalence in the convex case}
In the sequel, we suppose that $f\colon X\rightarrow \R\cup \{+\infty\}$ is a proper lower semicontinuous convex function. The following extra-properties are available .
\begin{itemize}
	\item $\partial^F f(x)=\partial f(x)=\{u\in X^*| \langle u,y-x\rangle\leq f(y)-f(x),\, \forall y\in X\},\, \forall x\in\dom f$.
	\item $|\nabla f| (x)= \dist(0,\partial f(x)), \forall x\in X.$	
\end{itemize}
In the convex case, Theorem~\ref{AzeCorLinear} can be simplified by the following proposition.
\begin{proposition} \cite{AzeCor02}\label{slope}
	 For $-\infty<\alpha<\beta\leq+\infty$, the following assertions hold true:
	\begin{description}
		\item[(i)]  $|\nabla f|(x)=\sup_{f(z)\leq f(x)} \frac{f(x)-f(z)}{\dist(x,z)}$, with $x$ is not a minimum point of $f$.
		\item[(ii)]$\inf_{[\alpha<f<\beta]} |\nabla f|(x) \geq \inf_{[f=\alpha]} |\nabla f|(x)$.
\item[(iii)]$\inf_{\alpha\leq \gamma<\beta}\left(\inf_{x\in [\gamma<f<\beta]}\frac{f(x)-\gamma}{\dist(x,[f\leq \gamma])}\right)=\inf_{x\in [\alpha<f<\beta]}\frac{f(x)-\alpha}{\dist(x,[f\leq \alpha])}.$
	\end{description}
\end{proposition}
Thanks to Proposition~\ref{slope}, the convex version of Theorem~\ref{Linear} is given as following.
\begin{theorem}\label{Linearconv}
Suppose $-\infty<\alpha<\beta\leq +\infty$ and $\tau>0$. Consider the following statements
	\begin{description}
		\item[(i)] $\displaystyle \inf_{x\in [f=\alpha]}\dist(0, \partial f(x)) \geq \frac{1}{\tau}$.
		\item[(ii)] $\displaystyle \inf_{x\in [\alpha<f<\beta]}\dist(0, \partial f(x)) \geq \frac{1}{\tau}.$
		\item[(iii)] $\tau \left(f(x)-\alpha\right)\geq \dist\left(x,[f(x)\leq\alpha]\right), \, \forall x\in [\alpha<f<\beta].$
		\end{description}
Then $(i)\Longrightarrow (ii)\Longleftrightarrow (iii)$.		
		\end{theorem}
We mention that the assumption (i) in the above theorem is equivalent to  the condition $0\notin\cl\left(\partial f\left(f^{-1}(0)\right)\right)$, which is called \textit{strong Slater's condition} \cite{LewPang, NgFeng1}.

We now consider the converse of Corollary~\ref{nonlinear}. Assume that 
$$\f (f(x)-\alpha)\geq \dist(x,[f\leq \alpha]),\,\forall x\in [\alpha<f<\beta], \, \f\in \KK (0,\beta-\alpha),$$
which is equivalent to
$$\frac{\f (f(x)-\alpha)}{f(x)-\alpha} \frac{f(x)-\alpha}{\dist(x,[f\leq \alpha])}\geq 1,\,\forall x\in [\alpha<f<\beta].$$
Thanks to Proposition\ref{slope}, the latter inequality implies that
$$\frac{\f (f(x)-\alpha)}{f(x)-\alpha} \dist(0,\partial f(x))\geq 1,\,\forall x\in [\alpha<f<\beta].$$
Thus, if $\f$ satisfies the condition
$$\int_{0}^{\beta-\alpha} \frac{\f(s)}{s} ds<+\infty,$$
then we get
$$\psi'(f(x)-\alpha) \dist(0,\partial f(x))\geq 1,\,\forall x\in [\alpha<f<\beta], $$
where $$\psi(s)=\int_{0}^{s} \frac{\f(t)}{t}dt,\,\forall s>0.$$
Therefore, when $f$ is convex, the converse of Corollary~\ref{nonlinear} is given as following. 
\begin{theorem}\label{c1:equiv}
Assume that $\varphi \left(f(x)-\alpha\right)\geq \dist\left(x,[f\leq \alpha]\right), \, \forall x\in [\alpha<f<\beta]$, where $\varphi\in \KK[0,\beta-\alpha)$ and
\begin{equation}\label{c1:condequiv}
\int_{0}^{\beta-\alpha} \frac{\f(s)}{s} ds<+\infty.
\end{equation}
Then, we get 
$$\psi'(f(x)-\alpha)\dist(0,\partial f(x))\geq 1,\,\forall x\in [\alpha<f<\beta].$$

\end{theorem}
This result has been appeared in \cite[Theorem 6]{eb&kl}, \cite[Theorem 30 ]{BolDanLeyMaz}. We remark that when $\f(s)=\tau s^\theta,\,(\tau,\,\theta>0)$, then the condition (\ref{c1:condequiv}) holds.

We will show that the Theorem~\ref{Linearconv} covers numerous results on Lipschitz global error bounds in the literature. 
\smallskip
 \begin{itemize}
 \item In \cite{Rob}, Robinson proved that if $f$ satisfies the Slater condition (there exists $\bar{x}$ such that $f(\bar{x})<0$) and the set $[f\leq 0]$ is bounded then $f$ has a Lipschitz global error bound. More generally, in \cite{Deng98}, Deng proved the following fact: 
 If there exist $\delta>0,\, \Delta>0$ such that
 $$[f\leq -\delta]\ne \emptyset \text{ and }  \sup_{[f\leq 0]} \dist(x,[f\leq -\delta])\leq \Delta, $$
then $$\dist(x,[f\leq 0])\leq \frac{\Delta}{\delta} [f(x)]_+, \:\: \forall x\in X.$$
Let us show that this result is actually a consequence of Theorem~\ref{Linearconv}. Indeed, take $x\in [f=0]$ and $u\in\partial f(x)$. For any $\varepsilon >0$, there exists $z\in [f\leq -\delta]$ such that $\dist(x,z)\leq \Delta+\varepsilon$. Thus, we obtain
 $$\delta\leq f(x)-f(z)\leq \|u\| \|x-z\|\leq \|u\| (\Delta+\varepsilon),$$
which implies that 
 $$\inf_{[f=0]} \dist(0,\partial f(x))\geq \frac{\delta}{\Delta}.$$ 
Combining with Theorem~\ref{Linearconv}, $f$ has Lipschitz global error bound.

Note that Deng's result \cite[Theorem 1]{Deng98} also covers the one in \cite{Deng97}, in which the author start form the assumption that there exist a unit vector $u$ and a constant $\tau>0$ such that 
\begin{equation}\label{c1:recession}
f^\infty(u)=\sup_{t>0} \frac{f(x+tu)-f(x)}{t}\leq -\frac{1}{\tau}
\end{equation}
to derive that $\dist(x,[f\leq 0])\leq \tau [f(x)]_+, \forall x\in X$.

\item The work of Robinson was also generalized in other directions. More precisely, instead of the boundedness assumption on the set $[f\leq 0]$, in \cite{Man}, Mangasarian used the asymptotic constraint qualification condition (this means {\em for any sequence $(x_k)_{k\in \N}\subset [f=0]$ such that $\lim\|x_k\|=\infty$, then the zero vector is not a limit point of any sequence $(u_k)_{k\in \N}$, with $ u_k\in \partial f(x_k)$}) to obtain a Lipschitz global error bound for differentiable convex function. Auslender and Crouzeix in \cite{Aus} extended the work of Mangasarian to the case nonsmooth convex functions. On the other hand, in \cite[Theorem 2]{Klatte}, Klatte and Li proved that, a convex function satisfies the Slater and the asymptotic qualification conditions if and only if
$$\inf_{x\in [f=0]} \dist(0,\partial f(x))> 0.$$
Therefore, it is clear that the results of Mangasarian \cite{Man}, Auslender and Crouzeix \cite{Aus} are the consequences of Theorem~\ref{Linearconv}.

\item  In \cite{LewPang}, Lewis and Pang characterized  Lipschitz error bounds using directional derivatives as follows. Let $f\colon\R^n\rightarrow \R\cup \{+\infty\}$ be a lower semicontinuous and convex function. They proved that the Lipschitz global error bound holds for  $f$:
  $$\dist(x,[f\leq 0])\leq \tau f(x), \quad \forall x\in \R^n,$$
   if and only if 
   $$f'(\bar{x},d)\geq \tau \|d\|,\, \forall \bar{x}\in [f=0],\, d\in N_{[f\leq 0]}(\bar{x}),$$
where the cone normal is defined by 
$N_S(\bar{x})=\{u\in X^*| \langle u,y-\bar{x}\rangle\leq 0, \forall y\in S \},\, \forall S\subset \R^n$.    
This result has been obtained by several other researchers, we can mention here the works of Ng and Zheng \cite{NgZheng}, \cite{NgFeng1} where they characterized  error bounds for lower semicontinous functions. 

Consider now \cite[Theorem 3.1]{NgFeng1}.

Suppose that $X$ is a reflexive Banach space. Then the following statements are equivalent
\begin{description}
	\item[(i)] $\dist(x,[f\leq 0])\leq \tau [f(x)]_{+}$, for all $x\in X$.
	\item[(ii)] For each $x\in [f=0]$, we get
	 $$\inf \{f'(x,d)| d\in N_{[f=0]} (x), \|d\|=1\}\geq \frac{1}{\tau} .$$
\item[(iii)] For each $x\in X\backslash [f\leq 0]$, there exists $d_x\in X, \|d_x\|=1$, such that 
$$f'(x,d_x)\leq -\frac{1}{\tau}.$$ 	 
\end{description} 
Let us prove that the (iii) above assertions are equivalent to (ii) of Theorem~\ref{Linearconv}. Assume that the assumption (iii) holds, then for all $x\in X\backslash [f\leq 0]$, we get
 $$|\nabla f|(x)\geq -f'(x,d)\geq \frac{1}{\tau}.$$
Conversely, suppose that $|\nabla f|(x)\geq \frac{1}{\tau}, \forall x\in X\backslash [f\leq 0]$. Take any $x\in X$, we get
 $$\dist(0,\partial f(x)) \geq \frac{1}{\tau}, \forall x\in X\backslash [f\leq 0],$$
hence there exists $d_x\in X$ such that
 $$-\tau=\inf\{\langle x^*,d\rangle | x^* \in X^*, \|x^*\|\leq \tau\} \geq \sup \{\langle u,d\rangle | u\in \partial f(x)\}=f'(x,d).$$
It follows that 
	$$f'(x,d_x)\leq -\frac{1}{\tau}.$$

Similarly, by setting $\f(s)=\tau s^\theta,\, (\tau, \, \theta>0)$, we can see that the following result of Ng and Zheng \cite{NgZheng} is also a consequence of  Theorem~\ref{c1:equiv}:

Let $X$ be a reflexive Banach space and $f\colon X\rightarrow \R\cup \{+\infty\}$ a continuous function. Suppose that for each $x\in X\backslash S$, there exists $d_x\in X, \|d_x\|=1$  and $\tau>0, \theta\in (0,1)$ such that 
$$f'(x,d_x)\leq -\tau f^{1-\theta}(x).$$
Then we get 
	$$\dist(x,[f\leq 0])\leq \tau [f(x)]_+^{\theta}, \forall x\in X.$$
 \end{itemize}
\subsection{Qualification conditions and error bounds}
\subsubsection{Slater's condition and error bounds}
We recall that if there exists $\bar{x}$ such that $f(\bar{x})<0$ then $f$ is said to satisfy the Slater condition. This condition plays an important role for the study of error bounds. The existence of the Lipschitz global error bound usually requires the convexity and the Slater condition. We consider the following example, which shows that for a convex function without the Slater condition, the Lipschitz global error bound may fail to hold.
\begin{ejem}\cite{LewPang}\label{ex1}
	$f(x,y)=x+\sqrt{x^2+y^2},\, (x,y)\in \R^2.$
\end{ejem}
It is easy to check that the function $f$ is convex, nonnegative on $\R^2$ and $[f=0]=\{(x,0)|x\leq 0\}$ has empty interior. Take the sequence $(z_k=(-k,1))_{k\in \N}$ then $f(z_k)$ converges to $0$ but $\dist(z_k,[f\leq 0])=1,\forall k\in\N$, so that there is not global error bound for $S$. 
%{\color{red} Why this example? what it is supposed to illustrate?}

As mentioned earlier, the Slater condition was used for the first time by Robinson \cite{Rob}.
\begin{theorem}\cite{Rob}
Let $f_1,\ldots, f_m $ be convex functions on $\R^n$ and assume that there is $\bar{x}$ such that $f_i(\bar x)<0, \ldots ,f_m(\bar{x})<0$. Then there exists $\tau>0$ such that
$$\dist(x,[f\leq 0])\leq \tau \|x-\bar{x}\|\sum_{i=1}^{m} [f_i(x)]_+, \, \forall x\in \R^n,$$ 
where $f(x)=\max_{i=1,\ldots,m} f_i(x)$.

 In additional, when $\{x\in \R^n|f_i(x)\leq 0, (i=1,\ldots, m) \}$ is bounded then there exists $\tau>0$ such that
  $$\dist(x,[f\leq 0])\leq \tau \sum_{i=1}^{m} [f_i(x)]_+, \, \forall x\in \R^n.$$ 
\end{theorem} 
As a consequence, we immediately deduce that the convex function systems $f_1\ldots, f_m$ has Lipschitz local error bound.
 
Luo and Luo \cite{LuoLuo} used the Slater condition to establish the Lipschitz global error bound for convex quadratic systems, this result has been extended by Pang and Wang \cite{Pangwang}. In general, the Slater condition is not sufficient to ensure that the global error bound holds, even if $f$ is a convex function. We consider the following example: 
\begin{ejem}\cite{Li10}\label{Slater:Ex}
	Let $f_1,f_2\colon\R^4\rightarrow \R$ be defined by $f_1(x)=x_1$ and
	$$f_2(x)=x_1^{16}+x_2^8+x_3^6+x_1^{2}x_2^{4}x_3^2+x^2_2x^4_3+x^4_1x^4_4+x^4_1x^6_2+x^2_1x^6_2+x^2_1+x^2_2+x_3^2-x_4,$$
	for all $x=(x_1,x_2,x_3,x_4)\in \R^4$. Define $f(x)=\max \{f_1(x),f_2(x)\},\, \forall x\in \R^4$.
\end{ejem}
We get the following properties, (see \cite{Li10}).
\begin{description}
	\item[(i)] $f_1, f_2$ are convex polynomial functions, therefore $f$ is convex.
	\item[(ii)] $f$ satisfies the Slater condition.
	\item[(iii)] For any $\alpha, \beta \in \R$ with $\alpha\leq \beta$, then
	$\sup_{x\in [f\leq\beta]} \dist(x,[f\leq \alpha])=+\infty.$
\end{description}
By taking $\alpha=0,\,\beta=1$ in the property (iii), we imply that there exists a sequence $(x_k)_{k\in \N}\subset [f\leq 1]$ such that $\dist(x_k,[f\leq 0])=+\infty$, this show that $f$ can not possess the H\"older global error bound.

However enhancing the assumptions  we can derive global error bounds from the Slater like condition. For instance, let $f$ be a lower semicontinuous, convex function on $\R^n$ which satisfies the Slater condition, then $f$ has Lipschitz gloabl error bound if one of the following assertions holds.
\begin{enumerate}
%\item $f$ is a polynomial function on $\R^n$, then $f$ admits a H\"older global error bound, see \cite[Theorem 4.1, Theorem 4.4]{Yang}, \cite[Theorem]{Li10}.
\item  $f$ can be expressed as maximum of finitely many bounded below convex polynomials function on $\R^n$, i.e: $f(x)=\max_{i=1,\ldots,d} f_i(x),\,\forall x\in\R^n$, where $f_i$ is a polynomial function on $\R^n$ with $\inf f_i>-\infty$, for all $i=1,\ldots,d$, see  \cite[Theorem 4.1]{Li10}.
\item $f$ is a separable function (in the sense that $f(x)=\sum_{i=1}^{n}f_i(x_i)$ where $x=(x_1,\ldots,x_n)$ and each $f_i$ is a lower semicontinuous function), see\cite[Theorem 4.1]{Li10}.
\item $f$ is well-posed ({\em for any sequence $\{x_k\}_{k\in \N}$ for which $\dist(0,\partial f(x_k))\rightarrow 0$ then $f(x_k)\rightarrow \inf_{X}f$}), see \cite[Corollary 1]{LewPang}.
\item $f$ satisfies the asymptotic qualification condition, see \cite{Aus}.
\end{enumerate}

 Notice that if $f$ is a convex function, then $f$ satisfies Slater condition if and only if $0\notin \partial f\left(f^{-1}(0)\right)$. %Lewis and Pang, in \cite{LewPang} (see also in the work of Mangasarian \cite{Manga98}) presented the {\textit{strong Slater condition}} $$0\notin \cl\left(\partial f\left(f^{-1}(0)\right)\right).$$Thanks to Theorem~\ref{Linearconv}, this is a sufficient condition to possess a Lipschitz global error bound when $f$ is convex.
 We can easily see that, if $f$ satisfies the Slater condition and the level set $[f\leq 0]$ is bounded then $f$ possesses the strong Slater condition. Furthermore, 
 in \cite{Klatte}, Klatte and Li proved that, for a convex function $f\colon\R^n\rightarrow \Rcupinf$ which satisfies the Slater condition, the following conditions are equivalent:
\begin{enumerate}
	\item The strong Slater condition holds. 
	\item The asymptotic qualification condition is satisfied.
	\item $\sup_{x\in [f=0]}\inf_{y\in [f<0]}\frac{\|x-y\|}{-f(y)}<+\infty.$
\end{enumerate}

\subsubsection{Abadie qualification condition and error bounds}

We begin this subsection by considering an example:
\begin{ejem}\cite{WLi97}\label{WLI97}
{\em	For $(x,y)\in\R^2$, take $f_1(x,y)=x+y,\, f_2(x,y)=-x-y, \, f_3(x,y)=(x+y)^2$ and $f(x,y)=(f_1, f_2,f_3)(x,y)$. Then $[f\leq 0]=\{(x,-x)|x\in \R^n\}$ has no interior point, but we can check that $\dist((x,y),[f\leq 0])\leq 2 \|[f(x,y)]_+\|$, for all $(x,y)\in \R^2.$	}	
\end{ejem}
This means that, the global error bound may be hold without the Slater condition. In \cite{WLi97}, Li used the Abadie qualification condition to characterize Lipschitz-type error bound for convex quadratic systems.

Recall that, the tangent cone of $S\subset \R^n$ is defined by
 $$T_S(x)=\left\lbrace y\in \R^n|\langle u,y\rangle\leq 0, \forall u\in N_S(\bar{x})\right\rbrace.$$
 Let us now recall the definition of Abadie's condition.
\begin{definition}\cite{WLi97}
	We say that the systems $f_1,f_2,\ldots,f_m\colon X\rightarrow \R$ satisfies the Abadie condition at $\bar{x}\in S=\{x\in X|f_i(x)\leq 0, i=1,\ldots,m\}$ if
	$$ T_S(\bar{x})=\{u\in X | \left\langle f'_i(\bar{x}),u\right\rangle \leq 0, \forall i\in I(\bar{x})\},$$
	where $I(\bar{x})= \{i: f_i(\bar{x})=0\}$.
	
	If this property holds at every point in $S$, then we say that the systems $f_1,f_2,\ldots,f_m$ satisfies the Abadie condition on $S$.
\end{definition}
When $X=\R^n$ and $f_1,\ldots,f_m$ are convex functions, we have the two following properties, see \cite{WLi97}. 
\begin{enumerate}
	\item The system $f_1,f_2,\ldots,f_m$ satisfies the Abadie condition at $\bar{x}\in S$ if
	and only if 
	$$ N_S(\bar{x})=\left\lbrace\sum_{i\in I(\bar{x})} \lambda_i f'_i(\bar{x})|\lambda_i \geq 0\right\rbrace.$$
	\item If there exists $x\in S$ such that $f_i(x)<0$ with $f_i$ is not affine function, for all $i=1,\ldots,m$ then the systems $f_1,f_2,\ldots,f_m$ satisfies Abadie's condition on $S$.
\end{enumerate}

Let us now give a necessary and sufficient condition for a convex quadratic system to have a Lipschitz-type global error bound, which was established by Li \cite[Theorem 4.2]{WLi97}.
\begin{theorem}\cite{WLi97}
	Let $f_1,f_2,\ldots,f_m$ be convex quadratic functions on $\R^n$ such that $S=\{x\in\R^n| f_i(x)\leq 0,(i=1\ldots,m)\}$ is nonempty. The following statements are equivalent
	\begin{description}
		\item[(i)] The system $(f_i)_{i=1\ldots,m}$ satisfies the Abadie condition on $S$.
		\item[(ii)] There exists $\tau>0$ such that 
		$$\dist(x,S)\leq \tau \sum_{i=1}^{m}[f_i(x)]_+,\, \forall x\in \R^n.$$
	\end{description}
\end{theorem} 
Later, in \cite[Theorem 6]{NgaiThera05}, Ngai and Th\'era extended this result in the Banach space. In which, $f_i\colon X\rightarrow \R,i=1,\ldots, m$ are defined by 
$$f_i(x)=\frac{1}{2} \langle A_ix,x\rangle+ \langle B_i,x\rangle +c_i,$$
where  $A_i\colon X\times X \rightarrow \R$ be a symmetric continuous bilinear and semi-definite positive, $B_i\in X^*$ and $c_i\in\R$, for $i=1,\ldots,m$. In this paper, Ngai and Th\'era also gave the relation between the Abadie condition and the Lipschitz local error bound for the convex function systems.
\begin{theorem}\cite{NgaiThera05}\label{Ngaithera5}
Let $f_1,\ldots ,f_m,$ be convex continuous functions on the neighborhood of $\bar{x}\in S=\left\lbrace x\in X|f_i(x)\leq 0,\,i=1,\ldots,m\right\rbrace$. Set $f=\max_{i=1,\ldots,m} f_i$.
\begin{description}
	\item[(i)] If there exist $\tau>0,\,\varepsilon>0$ such that 
	$$\dist(x,S)\leq \tau\sum_{i=1}^{m}[f_i(x)]_+,\,\forall x\in B_\varepsilon(\bar{x})\cap K,$$
	then the Abadie condition is satisfied on $B_\delta (\bar{x})\cap S$ for some $\delta>0$.
	\item[(ii)] If $f_i,\ldots,f_m$ are differentiable on $B_\delta (\bar{x})$, then the converse of part (i) holds. 
\end{description} 
\end{theorem}

%Notice that if $f_i, (i=1,\ldots,m)$ are quadratic convex functions then the assertion (ii) holds for all $x\in \R^n$. 
%\subsubsection{Asymptotic constraint qualification}

%\subsection{Characterization of global error bounds for semi-algebraic functions}

\section{Existence and quantitative results}

The first result on local error bound was deduced from the result of H\"ormander, in his work on the fundamental solution of partial differential equation.
\begin{theorem}[H\"omander, 1958] \cite{Hor}\label{Horm}
	Let $f$ be a polynomial function on $\R^n$. With the assumption that $[f\leq 0]$ is nonempty, there exist $\tau>0, \, a>0$ and $b\in\R $ such that 
	$$\dist(x,[f=0])\leq \tau\left(1+\|x\|\right)^{b} |f(x)|^{a}, \forall x\in \R^n.$$
\end{theorem}
This ``error bound" has an extra factor of $\left(1+\|x\|\right)^{b}$. One sees that we can remove this extra factor when restricting the error bound to a bounded region, in that case this local error bound for $f$ can be deduced from. Luo and Luo applied the above theorem to obtain the H\"older local error bound for polynomial function systems \cite[Theorem 2.2]{LuoLuo}, this result was extended for analytic systems, by Luo and Pang \cite[Theorem 2.2]{LuoPang}. Recently, Kurdyka and Spondzieja \cite[Corollary 10]{Kur14} showed that the exponents $a,b$ in Theorem~\ref{Horm} can be computed explicitly: 
$$b=2,\, a=\frac{1}{d(6d-3)^{n-1}}.$$

\paragraph{Result of \L ojasiewicz } A very general local error bound is deduced from the result of \L ojasiewicz, Theorem \ref{loja59}, if we take $\phi(x)=f(x)$ and $\psi(x)=\dist(x,[f\leq 0]) $, we get a local error bound result for subanalytic functions, also called \L ojasiewicz function inequality. It also includes a special case of the polynomial equation studied by H\"ormander.
\begin{theorem}\cite{Loja59}\label{loja}
Let $f\colon\R^n\rightarrow \R$ be a continuous subanalytic function. For any compact set  $K\subset\R^n$, there exist  $\tau>0, a>0$ such that 
	$$\dist(x,[f\leq 0])\leq \tau[f(x)]_+^a, \,\forall x\in K.$$
\end{theorem} 

With a direct application of the Theorem~\ref{loja} to a subanalytic system, we recover a result of Luo and Pang \cite[Theorem 2.2]{LuoPang}, in which they obtained the similar result for an analytic system:

\begin{theorem}\label{LuoPang}
Let $f_1,f_2,\ldots, f_r$ and $g_1,g_2,\ldots ,g_s$ be continuous subanalytic functions on $\R^n$, set
	$$S=\left\lbrace x\in \R^n| f_i(x)\leq 0,i=1,\ldots, r; g_j(x)=0, j=1,\ldots s \right\rbrace.$$
Then, for each compact set $K\subset \R^n$, there exist $\tau>0, a>0$ such that 
	$$\dist(x,S)\leq \tau\left( \|[f(x)]_+\| + \|g(x)\|\right)^a, \,\forall x\in K,$$
where $f(x)=\left(f_1(x),\ldots, f_r(x)\right)$,\, $g(x)=\left(g_1(x),\ldots, g_s(x)\right)$.
\end{theorem}
We mention that in all the above results of error bound, the H\"older exponent is not unknown, even in the result of Luo and Luo \cite{LuoLuo} for polynomial function systems.
%We mention that the inequality in Corollary~\ref{loja} may fail hold when $K$ is not compact. We are considering the local error bound for polynomial. First, we refer to the result of H\"ormander \cite{Hor}.{\color{red} Le resultat de Hormander est plus important que les precedents, d'ailleurs connaissait il les travaux de Lojasiewics?? Hormander est une medaille Field quand meme....}

\subsection{Local error bound for polynomial}
%{Quantitative results, estimation of \L ojasiewicz exponents}
We are now interested in the estimation of exponents within error bounds. First, we present the result of Gwozdziewicz \cite{Gwo}, in which a local quantitative error bound for a single real polynomial function with a isolated zero is provided.

For each $n, d\in \N $, we set
$$\kappa (n,d)=(d-1)^n+1 \text{ and } 
R(n,d)=\begin{cases}
1& \text{ if }d=1\\
d(3d-3)^{n-1}& \text{ if }d\geq 2.
\end{cases}$$
%\begin{theorem}\cite{Acukur}
%Let f be a real polynomial on $R^n$ with degree $d\in \N$. Suppose that $f(0)=0$ and $\nabla f(0)=0$. Then there exist constants $\tau, \varepsilon >0$ such that 
%$$\|x\|\leq \tau f(x)^{\frac{1}{R(n,d)}},$$
%for all $\|x\|\leq \varepsilon$.
%\end{theorem}

\begin{theorem}\cite{Gwo}\label{Gwo}
Let $f$ be a polynomial function on $\R^n$ with degree $d$. Assume that $x=0$ is an isolate zero of $f$, this means $f(0)=0$ and there is $\delta>0$ with $f(x)\ne 0$, for all $x\in B_\delta(0)\backslash \{0\} $. Then there exist positive constants $\tau, \varepsilon$ such that 
$$\|x\|\leq \tau |f(x)|^{\frac{1}{\kappa(n,d)}},$$
for all $\|x\|\leq \varepsilon$.
\end{theorem}

A similar result for polynomial function system was given by Koll\'ar in \cite{Kollar}.  
\begin{theorem}\cite{Kollar}\label{Kolla}
	Let $f_1,\ldots,f_m$ be some polynomial functions on $\R^n$ whose degrees do not exceed $d$. Set $f(x)=\max_{i=1,\ldots,m} f_i(x)$ for $x$ in $\R^n$. Assume that there is $\delta>0$ such that $f(x)=0$ and $f(x)\ne 0, \,\forall x\in B_\delta(0)\backslash \{0\}$. Then there exist $\tau, \varepsilon$ such that 
	$$\|x\|\leq \tau |f(x)|^{\frac{1}{d^n\beta(n-1)}}, \,\text{for all }x \text{ such that }\| \,x\|\leq \varepsilon,$$
	where $$\beta(n-1)=\begin{pmatrix}
	n-1\\
	[\frac{n-1}{2}]
	\end{pmatrix}.$$
\end{theorem}
Without the assumption of isolated zero point, Kurdyka and Spodzieja \cite[Corollary 4]{Kur14} (see also \cite{Son12}) obtained an error bound for a polynomial function with the H\" older exponent
$a=R^{-1}(n,d).$

To our knowledge, these are the first general results on error bounds  with some estimations of the exponent. Some applications of these above results can be found in \cite{Borwein, Li10, Li, LiMorPham, Ngai,LMNP}. 
%In \cite{Borwein}, by combining Theorem \ref{Gwo} and Theorem \ref{Kolla}, the authors provided a local error bounds for convex polynomial system:
%\begin{theorem}\cite{Borwein}
%Let $f_1,\ldots,f_m$ be some polynomial on $\R^n$ whose degrees do not exceed $d$. Set $S=\{x\in \R^n| f_i(x\leq 0, \forall i=1, \ldots, m)\}$ and
%$$J_0=\{i\in \{1,\ldots,m\}|f_i(x)=0, \forall x\in S \}, \quad J_1=\{1,\ldots,m\}\setminus J_0.$$	
%Then, for any $\bar{x}\in S$, there exist $c>0$, $\varepsilon>0$ such that
%$$\dist(x,S)\leq c\left(\max_{J_1}[f_i(x)]_++\left(\max_{J_0} [f_j(x)]_+\right)^\theta \right), \text{ for all } \|x-\bar{x}\|\leq \varepsilon,$$ 
%where $$\theta=\max\left\lbrace\frac{2}{\kappa(n,2d)},\frac{1}{d^n\beta(n-1)} \right\rbrace$$
%\end{theorem}

In \cite{LiMorPham}, Li, Mordukhovich and Pham, gave local error bounds for polynomial function systems in the nonconvex case, with exponents explicitly determined by the dimension of the underlying space and the degree of the involved polynomial functions. In this work, they obtained two results, one is based on \L ojasiewicz gradient inequality, and the other result is proved with a technique similar to that of Theorem~\ref{LuoPang}.
\begin{theorem}\cite{LiMorPham} \label{LMP1}. Let $f_1, \ldots, f_r$ and $g_1,\ldots, g_s$ be real polynomial functions on $\R^n$ with degree at most $d$, and let 
	$$S=\left\lbrace x\in \R^n| f_i(x)\leq 0,\colon g_j(x)=0 \right\rbrace.$$ 
Then for each $\bar{x}\in S$ there exist $\tau>0,\,\varepsilon>0$ such that 
\begin{equation}\label{loja1}
	\dist(x,S)\leq \tau\left( \sum\limits_{i=1}^{r}[f_i(x)]_++\sum\limits_{j=1}^{s}|g_i(x)|\right)^\frac{1}{R(n+r+s,d+1)},\text{ with } \|x-\bar{x}\|\leq \varepsilon.
\end{equation}
\end{theorem}
%{\color{red} Je ne comprends pas du tout cette phrase, a reprendre: The authors proved \L ojasiewicz inequality for maximum of finitely many polynomials and then combined with Corollary \ref{Expo} to obtain \ref{loja1}. Before beginning the proof (OF WHAT), let us recall a result of D'Acunto and Kurdyka \cite{Acukur}, which established \L ojasiewicz inequality for polynomial.}
Before beginning the proof of the latter theorem, let us recall a result of D'Acunto and Kurdyka \cite{Acukur}, which established \L ojasiewicz gradient inequality for polynomial function.
\begin{theorem}\cite{Acukur}\label{kur}
Let $f$ be a polynomial function with degree $d$, suppose that $f(0)=0$. There exists $c>0,\,\varepsilon>0$ such that
	$$\|\nabla f(x)\|\geq \tau |f(x)|^{1-\frac{1}{R(n,d)}},\text{ with } \|x\|\leq \varepsilon.$$
\end{theorem}

Now, we apply this result to establish the \L ojasiewicz gradient inequality for maximum of finitely many polynomial functions.

\begin{lemma}\label{Loja_poly}
Let $f(x)=\max_{i=1,\ldots,r} f_i(x)$ where $f_i$ are polynomial functions  on $\R^n$ whose degrees do not exceed $d$, and $\bar{x}\in \R^n$ with $f(\bar{x})=0$. Then, exist $c>0, \, \varepsilon >0$ such that 
	$$dist\left(0,\partial f(x)\right)\geq \tau|f(x)|^{1-\frac{1}{R(n+r-1,d+1)}},\text{ with } \|x-\bar{x}\|\leq \varepsilon.$$
\end{lemma}

\begin{proof}
Without loss of generality, suppose that $f_i(\bar{x})=0,\, i=1,\ldots ,r$. For each subset $I=\left\lbrace i_1, \ldots, i_q\right\rbrace \subset \left\lbrace 1, \ldots, r\right\rbrace$, we define the polynomial function $F_I\colon\R^{n+q-1}\rightarrow \R$ as following
	$$ F_I(x,\lambda)=
	\begin{cases}
	\sum\limits_{j=1}^{q-1} \lambda_j f_{i_j}(x)+\left(1- \sum\limits_{j=1}^{q-1} \lambda_j  \right) f_{i_q}(x) &\text{ if } q\geq 2\notag\\
	f_{i_1}(x) &\text{ if } q=1,\notag
	\end{cases}$$
where $\lambda=(\lambda_1,\ldots,\lambda_{q-1})\in \R^{q-1}$. It is clear that $F_I$ has degree at most $d+1$ and $F(\bar{x},\lambda)=0, \forall \lambda \in\R^{q-1}$. Set 
$$P=\left\lbrace \lambda\in \R^{q-1}|\lambda_j\geq 0, \sum\limits_{j=1}^{q-1} \leq 1 \right\rbrace.$$
$P$ is a compact set. For each $\bar{\lambda}\in P $, if $\nabla F_I(\bar{x},\bar{\lambda})=0$, then thanks to Theorem~\ref{kur}, there exit $\varepsilon_I>0,\, \tau_I>0$ such that 
\begin{equation}\label{loja3}
\|\nabla F(x,\lambda)\|\geq \tau_I |F_I(x,\lambda)|^{1-\frac{1}{R(n+q-1,d+1)}}, \text{ with }  \|\lambda-\bar{\lambda}\|\leq \varepsilon_I,\, \|x-\bar{x}\|\leq \varepsilon_I.
\end{equation}
In the other case, when $\nabla F_I(\bar{x},\bar{\lambda)}\ne 0$ then (\ref{loja3})  immediately holds. By the compactness of $P$, the inequality (\ref{loja3}) holds for all $\lambda \in P$. Set 
$$\tau=\min \left\lbrace \tau_I|I\subset \{i,\ldots, r\}, I\ne \emptyset \right\rbrace>0\text{ and } \varepsilon=\min \left\lbrace \varepsilon_I|I\subset \{i,\ldots, r\}, I\ne \emptyset \right\rbrace>0.$$
Take an arbitrary point $x\in \R^n$ such that $ \|x-\bar{x}\|\leq \varepsilon$ and $I(x)=\left\lbrace i|f_i(x)=f(x) \right\rbrace$, then there exist $\lambda_i\geq 0, i\in I(x)$ and $\sum\limits_{i\in I(x)} \lambda_i=1$ such that 
	$$dist\left(0,\partial f(x)\right)=\|\sum\limits_{i\in I(x)} \lambda_i \nabla f_i(x)\|.$$
On the other hand, for $i\in I(x)$, we have
	$$F_{I(x)}(x,\lambda)=\sum\limits_{i\in I(x)} \lambda_i f_i(x)=f(x)$$
and 
	$$\|\nabla F_{I(x)}(x,\lambda)\|=\|\sum\limits_{i\in I(x)} \lambda_i \nabla f_i(x)\|=dist\left(0,\partial f(x)\right).$$
By combining the above inequalities and (\ref{loja3}), we have the conclusion.

\end{proof}

\medskip

\noindent
We now provide the proof of Theorem \ref{LMP1}\\
{\bf Proof of Theorem \ref{LMP1}}
We consider the proof for $\bar{x}\in \bd (S)$. For any $e=(e_i)_{i=1,\ldots,s}\in \left\lbrace -1,1\right\rbrace^s $, define the function 
	$$f_e(x)=\max \left\lbrace 0, f_i(x),\ldots, f_r(x), e_1g_1(x),\ldots,e_s g_s(x)\right\rbrace, \, \forall x\in\R^n.$$
One can see that $f_e$ is the maximum of $r+s+1$ polynomial function with degree not exceed d, and $f_e(\bar{x})=0$. Applying Lemma~\ref{Loja_poly}, one obtains $\tau_e>0$ and $\varepsilon_e>0$ such that
		$$dist\left(0,\partial f_e(x)\right)\geq \tau_e |f_e(x)|^{1-\frac{1}{R(n+r+s,d+1)}}, \forall \|x-\bar{x}\|\leq \varepsilon_e.$$
Set 
$$\tau=\min \left\lbrace \tau_e|e\in \{-1,1\}^s\right\rbrace >0, \quad \varepsilon=\left\lbrace \varepsilon_e|e\in \{-1,1\}^s\right\rbrace>0,$$
and 
$$f(x)=\max \left\lbrace 0, f_i(x),\ldots, f_r(x), g_1(x),\ldots,g_s(x), -g_1(x),\ldots, -g_s(x)\right\rbrace $$
For any $x$ with $\|x-\bar{x}\|\leq \varepsilon$ and $f(x)>0$, then we can find $e\in \{-1,1\}^s$ such that $f(x)=f_e(x)$ and 
$dist\left(0,\partial f(x)\right)=dist\left(0,\partial f_e(x)\right)$.
Therefore, 
$$\dist\left(0,\partial f(x)\right)\geq \tau|f(x)|^{1-\frac{1}{R(n+r+s,d+1)}}, \forall \|x-\bar{x}\|\leq \varepsilon.$$
By applying Corollary~\ref{localnonlinear} with $\varphi(s)=s^{\frac{1}{R(n+r+s,d+1)}},\,\forall s>0$, we obtain the conclusion. $\hfill\Box$

By using the same technique as in\cite[Theorem 2.1]{LuoPang}, \cite[Theorem 2.2]{LuoLuo}, one can obtain  an  error bound whose exponent is different from Theorem \ref{LMP1}.
\begin{theorem}\cite{LiMorPham}\label{LMP2}
With the  assumptions of Theorem \ref{LMP1}, we have the following local error bound.
$$\dist(x,S)\leq \tau\left( \sum\limits_{i=1}^{r}[f_i(x)]_++\sum\limits_{j=1}^{s}|g_i(x)|\right)^\frac{2}{R(n+r,2d)}, \text{ with } \|x-\bar{x}\|\leq \varepsilon.$$
\end{theorem}

%More general, in \cite{LMNT}, the authors give a error bound for parametric polynomial systems with exponent explicitly.
%\begin{theorem}
%Let $f_i, g_j, h_s: H\times\R^m$ with ($i=1,\ldots, I, j=1,\ldots, J,\,s=1,\ldots S)$ be polynomials of degrees at most $d$, and $K$ is a compact set in $H$. Denote
%$$Y(x)=\{y\in \R^m|g_j(x,y)\leq 0, \, h_s(x,y)=0\},$$ 
%and 
%	$$S=\{x\in H| f_i(x,y)\leq 0, \, \forall y\in Y(x), \, i=1,\ldots, I\}$$
%Then there are numbers $c, \varepsilon >0$ such that we have LEB
 %$$\dist(x,S)\leq c \left[\max_{y\in Y(x), i=1,\ldots, I} f_i(x,y)\right]_+^\tau, \, \forall x\in K$$
 %where the exponent $\tau$ is defined by
 %$$\tau=\begin{cases}
 %	R(2n+(m+r+s)(n+1),d+2)^{-1}& \text{ if } L=1\notag\\
 %	R(2n+(m+r+s+2)(n+1),d+L+1)^{-1}&\text{ if } L\geq 2.\notag
 %\end{cases}$$ 
%\end{theorem} 
%\begin{theorem}
%Let $F(x)=\left( f_{ij}(x)\right)_{i,j=1,\ldots p}, x\in H $ be a symmetric polynomial matrix of order $p$ and $f(x)$ is largest eigenvalue of the $F(x)$. We denote $S=\left\lbrace F(x)\prec 0\right\rbrace$, then for any compact set $K\subset H$, there exists a constant $c>0$ such that 
%	$$\dist(x,S)\leq c [f(x)]_+^\frac{1}{R(2n+p(n+1)), d+3}, \forall x\in K.$$
%\end{theorem}
% % % % % % % % % % % % % % % % % % % % % % % % %
\subsection{Global error bounds for polynomial }
%Let us start with the firsr result about GEB, it is the well-known result of Hoffman, in the case of linear function. 
%This result plays a leading role in research on EB. After that, there are a lot of results on EB. 
%In this, the convexity plays an important role in GEB.

% As we know, there are not many results GEB in the case nonconvex. The class of $f$ in two cases is different.
%For this reason, we divide our discussion into two subsection: nonconvex and convex case. 
\subsubsection{Nonconvex case}
We begin this subsection by recalling the result of Luo and Sturm \cite{LuoSturm}. The authors established the global error bound for the zero set of a quadratic function.
\begin{theorem} \cite{LuoSturm} Let $f\colon \R^n \rightarrow \R$ be the quadratic function. There exists a constant $\tau >0$ such that
	$$\dist(x,[f=0])\leq \tau (|f(x)|+|f(x)|^{\frac{1}{2}}), \forall x\in \R^n.$$
\end{theorem}
This result is recovered by the works of \cite[Corollary 5]{NgFeng1}, \cite[Corollary 2]{Tiep}. Remark that this theorem does not until hold for an arbitrary polynomial, 
\begin{ejem}
	Let $f(x,y)=(xy-1)^2+(x-1)^2, \,\forall (x,y)\in \R^2$.
\end{ejem}  
One has $[f\leq 0]=\{(1,1)\}$. Consider the sequence $(x_k=\frac{1}{k},y_k=k)_{k\in \N}$, it is easy to check that
$$0< f(x_k,y_k)=\left(1-\frac{1}{k}\right)^2 < 1, \forall k\in \N\text{ and } d((x_k,y_k), [f\leq 0])\rightarrow +\infty (k\rightarrow +\infty),$$
therefore, $f$ does not possess H\"older global error bound.

However, when $f$ is a polynomial convex, this result was proved by Yang \cite{Yang}, and we present it in Theorem~\ref{Yang08}.
\medskip

Let us now present the characterization of global error bound for semi--algebraic, which is proved by Ha \cite{Vui}.

Suppose that $f\colon \R^n\rightarrow \Rcupinf$ has a H\"older global error bound, 
\begin{equation}\label{geb}
\dist(x,[f\leq 0])\leq  \tau\left( [f(x)]_+^{a}+[f(x)]_+^b\right), \quad \forall x\in\R^n.
\end{equation} 

We observe easily that for any sequence $(x_k)_{k\in \N}\subset \R^n$, two following assertions hold
\begin{description}
	\item[(i)]  If $f(x_k)\rightarrow 0$, then $\dist(x_k,[f\leq 0])\rightarrow 0$.
	\item [(ii)]  If $\dist(x_k,[f\leq 0])\rightarrow +\infty$, then $f(x_k)\rightarrow +\infty$.
\end{description}
Conversely, in \cite{Vui}, Ha proved that, for a polynomial function  which satisfies two above conditions, then it possesses H\"older global error bound. This result was extended for the class of continuous semi-algebraic functions, see \cite[Theorem 2]{Tiep}. The definition of the semi-algebraic function is well-known, we can see the one in \cite[Definition 1]{Tiep}.

%{\color{red}[ What is this a standing assumption?? Please explain] If we take the sequence $(x_k)_{k\in \N}\subset \R^n$ such that $f(x_k)\rightarrow 0^+$, then the $\dist(x_k,[f\leq 0])\rightarrow 0$. And if $\dist(x_k,[f\leq 0])\rightarrow \infty$, then $f(x_k)$ is also required to $+\infty$. }
\begin{theorem}[Characterization of global error bound for semi-algebraic] \cite{Vui, Tiep}\label{cond}
	Let $f\colon\R^n \rightarrow \R$ be a continuous semi-algebraic function. The following statements are equivalent:
	\begin{enumerate}
		\item For any sequence $(x_k)_{k\in \N}\in \R^n\setminus [f\leq 0]$ and $ \|x_k\|\rightarrow +\infty$, we have:
		\begin{description}
			\item[(i)] If $f(x_k)\rightarrow 0$ then $\dist(x_k,[f\leq 0])\rightarrow 0$.\\
			\item [(ii)] If $\dist(x_k,[f\leq 0])\rightarrow +\infty$ then $f(x_k)\rightarrow +\infty$.
		\end{description}
		\item There exist $\tau>0$ and $a, b>0$ such that
		$$\dist(x,[f\leq 0])\leq \tau\left([f(x)]_+^a+[f(x)]_+^b\right),\,\forall x\in\R^n.$$
	\end{enumerate}
\end{theorem}
\begin{proof}
	$(2)\Rightarrow (1)$ is obvious, we now prove the implication $(1)\Rightarrow (2)$. The proof is divided into two parts. Using (i), we  shall prove that an error bound holds on the neighborhood of $[f\leq 0]$, while by using (ii) we provide a bound for large $\dist(x,[f\leq 0])$.
	
	Assume (i) holds. Let us prove that there exist $\tau_1>0, a>0$ and $r >0$ such that 
	$$\dist(x,[f\leq 0])\leq \tau_1[f(x)]_+^a, \forall x\in [f\leq r].$$
	For $t\in \R$, put $\varphi (t)=\sup_{} \{ \dist(x,[f\leq 0]): f(x)=t\}$. It is a semi-algebraic function. Thanks to (i), there exists $r>0$ such that $\varphi (t)<\infty$ for all $t\in [0,r]$. We can choose $r$ sufficiently small such that $\varphi (t)$ is continuous and $\varphi(t)\ne 0$ on $(0,r]$. By using Puiseux Lemma:
	$$\varphi (t)=\tau t^a +0(t^a),\, (t\rightarrow 0).$$
	From the assumption (i), it can be seen that $\tau>0$, $a>0$. So there exist $r>0$ and $\tau_1>0$ such that $\varphi(t)\leq \tau_1 t^a$, for all $t\in [f\leq r]$. It means that
	$$\dist(x,[f\leq 0])\leq \tau_1[f(x)]_+^{a}, \forall x\in [f\leq r].$$	
	
	Using (ii), let us prove that there exist $\tau_2>0, b>0$ and $\delta >0$ such that 
	$$\dist(x,[f\leq 0])\leq \tau_2[f(x)]_+^{b}, \forall x\in [\delta <f].$$
	This conclusion is clear when $f$ is bounded from above. We assume thus that $\sup_{\R^n}f=\sup_{\R^n}\varphi=+\infty$. It appears that $\varphi (t)>0$ when $t$ is sufficiently large, so there exist $\tau>0$ and $b>0$ such that
	$$\varphi (t)=\tau t^{b}+0(t^{b}).$$ 
	This implies that there is $c\tau_2>0$, $R>0$ such that
	$$\dist(x,[f\leq 0])\leq \tau_2 [f(x)]_+^{b}, \forall x\in [R<f].$$ 
	It is easily seen that (ii) implies the existence of $M>0$ such that $\dist(x,S)<M$, for all $x\in [r<f<R]$. It gives $\dist(x,[f\leq 0])\leq \frac{M}{r^\alpha}f(x)^\alpha $. Combining with such inequality on the domain $[f\leq r]$ and $[f\geq R]$, we have the conclusion.
\end{proof} 
\medskip

The implication $(2)\Rightarrow (1)$ in the latter theorem explains why do we need two exponents $[f(x)]_+^a$ and $[f(x)]_+^b$  in the global error bound (\ref{geb}). One is ensures that the inequality (\ref{geb}) holds when $\dist(x,[f\leq 0])\rightarrow 0$, and the other keeps such inequality holds when  $\dist(x,[f\leq 0])\rightarrow+\infty$. Generally, the exponents are different. 
\begin{ejem}\cite{SonVui}
	Let $f(x,y)=x^2+y^4,\,\forall (x,y)\in \R^2$.
\end{ejem}
It can be seen that $[f\leq 0]=\left\lbrace\right (0,0)\rbrace$, and
$$\dist((x,y),[f\leq 0])\leq  f^{\frac{1}{4}}(x,y)+f^{\frac{1}{2}}(x,y),\,\forall (x,y)\in \R^2.$$
On the other hands, by taking two sequences  $ (x^1_k=k,y^1_n=0)_{k\in\N}$ and $(x^2_k=0,y^2_k=1/k)_{k\in\N}$, this follows that there does not exist $\alpha \in \R$ such that 
$$\dist((x,y),[f\leq 0])\leq \tau [f(x,y)]_+^{\alpha}, \,\forall(x,y)\in \R^2.$$
By using Theorem \ref{cond}, Ha~\cite{Vui} provided a global error bound for polynomial function under a Palais--Smale condition. After that, his result was improved in \cite{Tiep} for  continuous semi-algebraic functions.

We recall that, $f$ is said to possess the Palais-Smale condition (PS) at $r_0$ if any sequence $(x_k)_{k\in \N}$, for which $f(x_k)\rightarrow r_0$ and $\dist\left(0, \partial f(x_k)\right)\rightarrow 0$, then $(x_k)_{k\in \N}$ possesses a converging subsequence.
\begin{theorem}\cite{Vui,Tiep} 
Let $f\colon \R^n\rightarrow \R$ be a continuous semi-algebraic function. Suppose that $f$ satisfies the Palais-Smale condition at each $r>0$, then there exist constants $\tau>0$ and $a, b>0$ such that
	$$\dist(x,[f\leq 0])\leq \tau\left( [f(x)]_+^a +[f(x)]_+^b \right), \forall x\in \R^n.$$
\end{theorem} 
\begin{proof}
It is enough to show that $f$ satisfies the two conditions (i) and (ii) in Theorem \ref{cond}. First we establish (i). By contradiction, we assume that there exists a sequence $(x_k)_{k\in \N}$ and a constant $\delta>0$ such that: 
	$$\|x_k\|\rightarrow \infty, f(x_k)\rightarrow 0 \text{ and } \dist(x_k,[f\leq 0])>\delta.$$
Put $X=\left\lbrace x|f(x)\geq 0 \right\rbrace$, then $X$ is a complete metric space. Applying Ekeland's  principle (see \cite{Eke}), there is a sequence $(y_k)_{k\in \N}\subset X$ such that
	$$f(y_k)\leq f(x_k)=\varepsilon_k$$
	$$\dist(x_k,y_k)\leq \sqrt{\varepsilon_k}$$
	$$f(y_k)\leq f(x)+\sqrt{\varepsilon _k}\dist(x,y_k), \forall x\in X.$$
It is clear that $f(y_k)\rightarrow 0$ and $\|y_k\|\rightarrow +\infty$. We can suppose that $\dist(y_k,[f\leq 0])\geq \frac{\delta}{2}$, therefore $\forall t\in(0,\frac{\delta}{2})$ and for all $ u\in \R^n, \|u\|=1$ we obtain
	$$\frac{f(y_k+tu)-f(y_k)}{t}\geq -\sqrt{\varepsilon_k}.$$ 
Thus $|\nabla f|(y_k)\leq \sqrt{\varepsilon_k} $. On the other hands,  $\|\partial f(y_k)\|\leq |\nabla f|(y_k)$, (see \cite[Remark 6.1]{AzeCor14}), therefore $\partial f(y_k)\rightarrow 0$, which is in contradiction with Palais-Smale's condition.

Now, we will prove that $f$ satisfies the condition (ii) of Theorem \ref{cond}. By contradiction, suppose that there exists a sequence  $(x_k)_{k\in\N}\subset \R^n$ such that: 
	$$\|x_k\|\rightarrow \infty,\: \dist(x_k, [f\leq 0])\rightarrow +\infty\text{ and } f(x_k)\rightarrow t\in \R.$$
Set $X=\left\lbrace x|f(x)\geq 0 \right\rbrace$,  $X$ is a complete metric space. Applying Ekeland's  principle, there is a sequence $(y_k)_{k\in\N}\subset X$ such that
	$$f(y_k)\leq f(x_k)=t_k$$
	$$\dist(x_k,y_k)\leq \frac{\dist(x_k,[f\leq 0])}{2}$$
	$$f(y_k)\leq f(x)+\frac{2f(x_k)}{\dist(x_k,[f\leq 0])} \dist(x,y_k), \forall x\in X$$
Therefore, without loss of generality we can assume that the sequence $f(y_k)$ is convergent, $\|y_k\|\rightarrow \infty$ and $\dist(y_k,[f\leq 0])\rightarrow +\infty$, therefore,
$$\|\nabla f(y_k)\|\leq |\nabla f|(y_k)\leq\frac{2f(x_k)}{\dist(x_k,[f\leq 0])} \rightarrow 0,$$
 contradicting to Palais-Smale's condition. 
\end{proof}

\medskip
 
\subsubsection{Convex case}  
We begin this subsection by giving a result of Facchinei, Pang \cite{FaccPang}, they assert that a lower semicontinuous convex function, a H\"older-type error bound on a level set can be extended to a global error bound.
\begin{theorem}\cite{FaccPang}\label{FacPang} 
Let $f$ be a lower semicontinuous convex function on $\R^n$ with $[f\leq 0]$ nonempty. Suppose that there exist $\delta>0$ and $\tau>0,\,\theta >0$ such that
 $$\dist(x,[f\leq 0])\leq  \tau\left([f(x)]_++[f(x)]^\theta_+\right), \quad \forall x\in[f\leq \delta].$$
There exists $\tau'>0$ such that 
$$\dist(x,[f\leq 0])\leq \tau'\left([f(x)]_+ +[f(x)]^\theta_+\right), \quad \forall x\in \R^n.$$ 
When we take $\theta=1$, this means that for a convex function, the Lipschitz error bound on the level set can be extended to a global error bound.
\end{theorem}
\begin{proof}
	Let $x\in \R^n$ such that $f(x)>\delta$ and $p=P_{[f\leq 0]}x$. It is clear that $f(p)=0$. For any $\lambda\in (0,1)$, we denote $x_\lambda=\lambda x+(1-\lambda)p$. It can be seen that $p=P_{[f\leq 0]}x_\lambda$ and $\dist(x_\lambda,[f\leq 0])=\lambda \dist(x,[f\leq 0])$. By convexity, we get 
	$$f(x_\lambda)\leq \lambda f(x)+(1-\lambda)f(p)=\lambda f(x).$$
We deduce that
   $$\dist(x,[f\leq 0])\leq \frac{\dist(x_\lambda,[f\leq 0])}{f(x_\lambda)} f(x).$$
On the other hand, by choosing $\lambda=\frac{\delta}{2f(x)}$, we get 
 $$f(x_\lambda)\leq \lambda f(x)=\frac{\delta}{2}<\delta.$$
 Therefore, thanks to the assumption on error bounds, we obtain 
 $$\dist(x_\lambda,[f\leq 0])\leq \tau\left(f(x_\lambda)+f^\theta(x_\lambda)\right).$$
 It follows that
 $$\frac{\dist(x_\lambda,[f\leq 0])}{f(x_\lambda)}<\tau\left(1+f^{\theta-1}(x_\lambda)\right)<c\left(1+\left(\frac{\delta}{2}\right)^{\theta-1}\right) .$$
 Combining the above inequalities, we get
 $$\dist(x,[f\leq 0])\leq \tau\left(1+\left(\frac{\delta}{2}\right)^{\theta-1}\right)f(x).$$
 This means that
 $$\dist(x,[f\leq 0])\leq \tau\left(1+\left(\frac{\delta}{2}\right)^{\theta-1}\right) \left(f(x)+f^\theta(x)\right), \forall x\in \R^n.$$
\end{proof}

\medskip

\noindent
Combining this result with Theorem~\ref{loja}, we immediately obtain a result similar to \cite[Theorem 3]{eb&kl} and \cite[Theorem 6]{deng1998perturbation}.
\begin{theorem}
	Let $f_i\colon\R^n \rightarrow \R,\, (i=1,\ldots m)$ be continuous, convex and subanalytic functions. Assume that, the set 
	$$S=\left\lbrace x\in\R^n|f_i(x)\leq 0,\, i=1,\ldots m\right\rbrace$$
is nonempty, compact. Then, there exist $\tau,\,\theta>0$ such that
	$$\dist(x,S)\leq \tau\left([f(x)]_++[f(x)]_+^\theta\right),\, \forall x\in \R^n,$$
where $f(x)=\sum_{i=1}^{m}[f_i(x)]_+$.
\end{theorem}
We remark that if $f_i$ is coercive then for all $r\in \R$, the set $[f_i\leq r]$ is compact.
%We are considering the example, where the global error bound may be fail for the convex function.
%\begin{ejem}\cite{LewPang}\label{ex1}
%	$f(x,y)=x+\sqrt{x^2+y^2}$
%\end{ejem}
%It is easy to check that $S=\{(x,0)|x\leq 0\}$ has nonempty interior. We take the sequence $(z_k=(-k,1))_{k\in \N}$ then $f(z_k)$ converges to $0$ but $\dist(z_k,S)=1,\forall k\in\N$, so that there is not global error bound for $S$. 

%So, what is the condition ensuring that the global error bound holds for convex function? We find out this example two properties\colon One is that $S$ has nonempty interior and the other is that $\inf_{(x,y)\in \R^2}\|f'(x,y)\| =0$. Now let us introduce some results on the necessary and sufficient condition to global error bound for convex function.  
%\begin{ejem}\cite{LewPang}\label{ex2}
%	$$f(x,y)=\begin{cases}
%		\frac{x^2}{y}& \text{ if } y> 0\notag\\
%		0&\text{ if } x=y=0\notag\\
%		\infty& \text{ otherwise}\notag
%	\end{cases}$$
%\end{ejem}
%\begin{theorem} \cite{AzeCor}
%Let $f\colon \mathbb{R}^n \longrightarrow \mathbb{R}$ be a lower semicontinuous proper convex function. Then the following statements are equivalent\colon
%\begin{enumerate}
%	\item $f$ admits a H\{"}oder-tyle global error bound: Exists $a>0, \alpha\in (0,1]$ such that 
%	$$\dist(x,S)\leq c[f(x)]^\alpha_+, \forall x\in \mathbb{R}^n.$$	
%	\item  There exist $b>0$ such that $\inf \{\partial f(x)| f(x)>0\}\geq [f(x)]_+^{1-\alpha}$. 
%\end{enumerate}
%\end{theorem}
We recall now the definition of piecewise convex polynomial functions.
\begin{definition}\cite{WLi95,Li}
A continuous function $f$ on $\R^n$ is called to be a piecewise convex polynomial function if there exist finitely many polyhedra $P_1,\ldots, P_k$ with $\cup_{j=1}^kP_j=\R^n$ such that the restriction of $f$ on each $P_j$, denoted by $f_i$, is a convex polynomial function. The degree of $f$, denoted by $\deg (f)$, is defined as the maximum of $\deg (f_j)$.
\end{definition}
In \cite{WLi95}, Li studied error bounds for a convex piecewise quadratic function. More precisely, let $f$ be a convex piecewise quadratic function. Then, there exists $\tau>0$ such that
\begin{equation}\label{Li95}
\dist(x,[f\leq 0])\leq \tau\left([f(x)]_++\sqrt{[f(x)]_+} \right), \forall x\in \R^n.
\end{equation}
By using Theorem \ref{Kolla} and Theorem \ref{FacPang}, Li \cite{Li10} showed that, for a convex polynomial function $f$ on $\R^n$ with degree $d$, there exists $\tau>0$ such that
 	\begin{equation}\label{Li10}
 	\dist(x,[f\leq 0])\leq \tau\left( [f(x)]_++[f(x)]_+^\frac{1}{\kappa(n,d)}\right), \forall x\in \R^n.
 	\end{equation}
This result is further improved by Yang \cite{Yang}. 
\begin{theorem}\cite{Yang}\label{Yang08}
Let $f$ be a polynomial convex with degree $d$. There exists $\tau >0$ such that 
$$\dist(x,[f\leq 0])\leq \tau\left( [f(x)]_++[f(x)]_+^\frac{1}{d}\right), \forall x\in \R^n.$$ 
\end{theorem}

%Li showed that, for a convex polynomial function $f$ on  $\R^n$ with degree $d$, there exists $c>0$ such that

The two above results (\ref{Li95}), (\ref{Li10}) have been extended by Li (\cite{Li}), for general convex piecewise polynomial function.
\begin{theorem}\cite{Li}
 	Let f be a piecewise convex polynomial function on $\R^n$ with degree $d$. Suppose that one of the following two conditions holds:
 	\begin{itemize}
 		\item[(i)]If  $\dist(x,[f\leq 0])\rightarrow +\infty$  then $f(x)\rightarrow +\infty $.
 		\item[(ii)] $f$ is convex.
 	\end{itemize}
 	There exists $c>0$ such that
 	$$\dist(x,[f\leq 0])\leq c\left( [f(x)]_++[f(x)]_+^\frac{1}{\kappa(n,d)}\right), \forall x\in \R^n.$$	
 \end{theorem}
 Let us now present a global error bound for convex polynomial function systems. In \cite{LuoLuo}, under the Slater condition, Luo and Luo proved that a global Lipschitzian error bound holds for convex quadratic systems. After that, without the Slater condition, Pang and Wang in \cite{Pangwang}, showed that any systems of convex quadratic has a global error bound. 
\begin{theorem}\cite{Pangwang}
Let $f_1, f_2 \ldots, f_m$ be convex quadratic functions. Assume that 
	$$S=\left\lbrace x\in \R^n|f_i(x)\leq 0, i=1,\ldots m \right\rbrace$$
is not empty, then there exists a positive integer $\dist\leq n+1$ and a scalar $c>0$ such that 
	$$\dist(x,S)\leq c\max \left( \|[f(x)]_+\|, \|[f(x)]_+\|^\frac{1}{2^d}\right), \forall x\in \R^n,$$
where $f(x)=(f_i(x))_{i=1\ldots, m}, \, \forall x\in \R^n$. 

\smallskip

Furthermore, if $S$ contains an interior point, then $d=0$. 
\end{theorem}
Similarly Theorem~\ref{Ngaithera5}, the latter result is extended to the Banach space in \cite[Theorem 7]{NgaiThera05}, with  
$$f_i(x)=\frac{1}{2} \langle A_ix,x\rangle+ \langle B_i,x\rangle +c_i,$$
where  $A_i\colon X\times X \rightarrow \R$ is a symmetric continuous bilinear and semi-definite positive, $B_i\in X^*$ and $c_i\in\R$, for $i=1,\ldots,m$.

Note that this result does not hold for a general convex polynomial function system, see Example~\ref{Slater:Ex}. However, in some particular cases, the global error bound hold for such systems.
\begin{theorem}
Let $f_1,\ldots,f_p$ be convex polynomial functions on $\R^n$ whose degrees are at most $d$. Let $f(x)=\max_{i=1,\ldots,m} f_i(x),\,\forall x\in\R^n$. Then, the following statements are hold
	\begin{enumerate}
\item \cite{Li10} If $f_i(x)\geq 0, \forall x\in \R^n, i=1,\ldots,m$ then there exists $\tau>0$ such that
$$\dist(x,[f\leq 0])\leq \tau\left( [f(x)]_++[f(x)]_+^\frac{1}{\kappa(n,d)}\right), \forall x\in \R^n.$$
\item \cite{Ngai} If $S=\{x\in K| f(x)\leq 0 \}$ is a nonempty compact set, where $K$ is a convex polyhedral in $\R^n$. Then, there exists $\tau>0$ such that
		$$\dist(x,S)\leq c\left( [f(x)]_++[f(x)]_+^\frac{1}{\kappa(n,2d)}\right), \forall x\in K.$$
\item \cite{Ngai}Let  $K$ is a convex polyhedral in $\R^n$ and $S=\{x\in K| f(x)\leq 0 \}$ is  nonempty. Assume that, for each $v\in K^\infty$: $\max_{i=1,\ldots,p} f_i^\infty (v)=0\Rightarrow f_i^\infty(v)=0$ (see (\ref{c1:recession})). Then, there exists $\tau>0$ such that
$$\dist(x,S)\leq c\left( [f(x)]_++[f(x)]_+^\frac{1}{\kappa(n,2d)}\right), \forall x\in K,$$
where $K^\infty$ is recession cone of $K$, defined by 
$$K^\infty=\left\lbrace v\in\R^n| x+tv\in K, \forall t>0, x\in K \right\rbrace.$$
\end{enumerate}
\end{theorem}

\bibliographystyle{plain}
\bibliography{Bibli}

\begin{thebibliography}{10}

\bibitem{AbsMahAnd}
Pierre~Antoine Absil, Robert Mahony, and Benjamin Andrews.
\newblock Convergence of the iterates of descent methods for analytic cost
  functions.
\newblock {\em SIAM Journal on Optimization}, 16(2):531--547, 2005.

\bibitem{attbol}
Hedy Attouch and J{\'e}r{\^o}me Bolte.
\newblock On the convergence of the proximal algorithm for nonsmooth functions
  involving analytic features.
\newblock {\em Mathematical Programming}, 116(1):5--16, 2009.

\bibitem{AttBolRedSou}
Hedy Attouch, J{\'e}r{\^o}me Bolte, Patrick Redont, and Antoine Soubeyran.
\newblock Proximal alternating minimization and projection methods for
  nonconvex problems: An approach based on the kurdyka--{\ l}ojasiewicz
  inequality.
\newblock {\em Mathematics of Operations Research}, 35(2):438--457, 2010.

\bibitem{AttBolSva}
Hedy Attouch, J{\'e}r{\^o}me Bolte, and Benar~Fux Svaiter.
\newblock Convergence of descent methods for semi-algebraic and tame problems:
  proximal algorithms, forward--backward splitting, and regularized
  gauss--seidel methods.
\newblock {\em Mathematical Programming}, 137(1-2):91--129, 2013.

\bibitem{Aus}
Alfred Auslender and Jean~Pierre Crouzeix.
\newblock Global regularity theorems.
\newblock {\em Mathematics of Operations Research}, 13(2):243--253, 1988.

\bibitem{Aze}
Dominique Az{\'e}.
\newblock A survey on error bounds for lower semicontinuous functions.
\newblock In {\em ESAIM: Proceedings}, volume~13, pages 1--17. EDP Sciences,
  2003.

\bibitem{AzeCor2}
Dominique Az{\'e} and Jean~No{\"e}l Corvellec.
\newblock On the sensitivity analysis of hoffman constants for systems of
  linear inequalities.
\newblock {\em SIAM Journal on Optimization}, 12(4):913--927, 2002.

\bibitem{AzeCor02}
Dominique Az{\'e} and Jean~No{\"e}l Corvellec.
\newblock Characterizations of error bounds for lower semicontinuous functions
  on metric spaces.
\newblock {\em ESAIM: Control, Optimisation and Calculus of Variations},
  10(3):409--425, 2004.

\bibitem{AzeCor14}
Dominique Az{\'e} and Jean~No{\"e}l Corvellec.
\newblock Nonlinear local error bounds via a change of metric.
\newblock {\em Journal of Fixed Point Theory and Applications},
  16(1-2):351--372, 2014.

\bibitem{AzeCor16}
Dominique Az{\'e} and Jean~No{\"e}l Corvellec.
\newblock Nonlinear error bounds via a change of function.
\newblock {\em Journal of Optimization Theory and Applications}, pages 1--24,
  2016.

\bibitem{amirbeck}
Amir Beck and Shimrit Shtern.
\newblock Linearly convergent away-step conditional gradient for non-strongly
  convex functions.
\newblock {\em Mathematical Programming}, pages 1--27, 2015.

\bibitem{BeckTeboulle}
Amir Beck and Marc Teboulle.
\newblock Convergence rate analysis and error bounds for projection algorithms
  in convex feasibility problems.
\newblock {\em Optimization Methods and Software}, 18(4):377--394, 2003.

\bibitem{BolDanLew1}
J{\'e}r{\^o}me Bolte, Aris Daniilidis, and Adrian Lewis.
\newblock The {\l}ojasiewicz inequality for nonsmooth subanalytic functions
  with applications to subgradient dynamical systems.
\newblock {\em SIAM Journal on Optimization}, 17(4):1205--1223, 2007.

\bibitem{BolDanLewShi07}
J{\'e}r{\^o}me Bolte, Aris Daniilidis, Adrian Lewis, and Masahiro Shiota.
\newblock Clarke subgradients of stratifiable functions.
\newblock {\em SIAM Journal on Optimization}, 18(2):556--572, 2007.

\bibitem{BolDanLeyMaz}
J{\'e}r{\^o}me Bolte, Aris Daniilidis, Olivier Ley, and Laurent Mazet.
\newblock Characterizations of {\l}ojasiewicz inequalities: subgradient flows,
  talweg, convexity.
\newblock {\em Transactions of the American Mathematical Society},
  362(6):3319--3363, 2010.

\bibitem{eb&kl}
J{\'e}r{\^o}me Bolte, Trong~Phong Nguyen, Juan Peypouquet, and Bruce Suter.
\newblock From error bounds to the complexity of first-order descent methods
  for convex functions.
\newblock {\em Mathematical Programming}, pages 1--37, 2015.

\bibitem{BST}
J{\'e}r{\^o}me Bolte, Shoham Sabach, and Marc Teboulle.
\newblock Proximal alternating linearized minimization for nonconvex and
  nonsmooth problems.
\newblock {\em Mathematical Programming}, 146(1-2):459--494, 2014.

\bibitem{Borwein}
Jonathan Borwein, Guoyin Li, and Liangjin Yao.
\newblock Analysis of the convergence rate for the cyclic projection algorithm
  applied to basic semialgebraic convex sets.
\newblock {\em SIAM Journal on Optimization}, 24(1):498--527, 2014.

\bibitem{corv}
Jean~No{\"e}l Corvellec and Viorica Motreanu.
\newblock Nonlinear error bounds for lower semicontinuous functions on metric
  spaces.
\newblock {\em Mathematical Programming}, 114(2):291, 2008.

\bibitem{Acukur}
Didier D'Acunto and Krzysztof Kurdyka.
\newblock Explicit bounds for the {\l}ojasiewicz exponent in the gradient
  inequality for polynomials.
\newblock {\em Annales Polonici Mathematici}, 87(1):51--61, 2005.

\bibitem{Dedieu00}
Jean~Pierre Dedieu.
\newblock Approximate solutions of analytic inequality systems.
\newblock {\em SIAM Journal on Optimization}, 11(2):411--425, 2000.

\bibitem{Deng97}
Sien Deng.
\newblock Computable error bounds for convex inequality systems in reflexive
  banach spaces.
\newblock {\em SIAM Journal on Optimization}, 7(1):274--279, 1997.

\bibitem{Deng98}
Sien Deng.
\newblock Global error bounds for convex inequality systems in banach spaces.
\newblock {\em SIAM journal on control and optimization}, 36(4):1240--1249,
  1998.

\bibitem{deng1998perturbation}
Sien Deng.
\newblock Perturbation analysis of a condition number for convex inequality
  systems and global error bounds for analytic systems.
\newblock {\em Math. Program.}, 83:263--276, 1998.

\bibitem{Tiep}
Si~Tiep Dinh, Huy~Vui Ha, and Tien~Son Pham.
\newblock H\" older-type global error bounds for non-degenerate polynomial
  systems.
\newblock {\em arXiv preprint arXiv:1411.0859}, 2014.

\bibitem{DruLew}
Dmitriy Drusvyatskiy and Adrian Lewis.
\newblock Error bounds, quadratic growth, and linear convergence of proximal
  methods.
\newblock {\em arXiv preprint arXiv:1602.06661}, 2016.

\bibitem{Eke}
Ivar Ekeland et~al.
\newblock Nonconvex minimization problems.
\newblock {\em Bull. AMS}, 1(3), 1979.

\bibitem{FaccPang}
Francisco Facchinei and Jong~Shi Pang.
\newblock {\em Finite-dimensional variational inequalities and complementarity
  problems}.
\newblock Springer Science \& Business Media, 2007.

\bibitem{PierreGuiJuan}
Pierre Frankel, Guillaume Garrigos, and Juan Peypouquet.
\newblock Splitting methods with variable metric for kl functions.
\newblock {\em arXiv preprint arXiv:1405.1357}, 2014.

\bibitem{Gwo}
Janusz Gwo{\'z}dziewicz.
\newblock The {\l}ojasiewicz exponent of an analytic function at an isolated
  zero.
\newblock {\em Commentarii Mathematici Helvetici}, 74(3):364--375, 1999.

\bibitem{Vui}
Huy~Vui Ha.
\newblock Global holderian error bound for nondegenerate polynomials.
\newblock {\em SIAM Journal on Optimization}, 23(2):917--933, 2013.

\bibitem{SonVui}
Huy~Vui Ha and Tien~Son Pham.
\newblock {\em Genericity in Polynomial Optimization}, volume~3.
\newblock World Scientific, 2016.

\bibitem{Hiro}
Heisuke Hironaka.
\newblock {\em Introduction to real-analytic sets and real-analytic maps}.
\newblock Istituto matematico" L. Tonelli" dell Universita di Pisa, 1973.

\bibitem{Hof}
Alan Hoffman.
\newblock On approximate solutions of systems of linear inequalities.
\newblock {\em Selected Papers Of Alan J Hoffman: With Commentary}, pages
  174--176, 2003.

\bibitem{Hor}
Lars H{\"o}rmander.
\newblock On the division of distributions by polynomials.
\newblock {\em Arkiv f{\"o}r matematik}, 3(6):555--568, 1958.

\bibitem{Iof00}
Aleksandr~Davidovich Ioffe.
\newblock Metric regularity and subdifferential calculus.
\newblock {\em Russian Mathematical Surveys}, 55(3):501--558, 2000.

\bibitem{Iof79}
Alexander~Davidovich Ioffe.
\newblock Regular points of lipschitz functions.
\newblock {\em Transactions of the American Mathematical Society}, 251:61--69,
  1979.

\bibitem{Klatte}
Diethard Klatte and Wu~Li.
\newblock Asymptotic constraint qualifications and global error bounds for
  convex inequalities.
\newblock {\em Mathematical programming}, 84(1):137--160, 1999.

\bibitem{Kollar}
J{\'a}nos Koll{\'a}r.
\newblock An effective {\l}ojasiewicz inequality for real polynomials.
\newblock {\em Periodica Mathematica Hungarica}, 38(3):213--221, 1999.

\bibitem{Kru15}
Alexander Kruger.
\newblock Error bounds and metric subregularity.
\newblock {\em Optimization}, 64(1):49--79, 2015.

\bibitem{Kur98}
Krzysztof Kurdyka.
\newblock On gradients of functions definable in o-minimal structures.
\newblock {\em Annales de l'institut Fourier}, 48(3):769--783, 1998.

\bibitem{Kur14}
Krzysztof Kurdyka and Stanis{\l}aw Spodzieja.
\newblock Separation of real algebraic sets and the {\l}ojasiewicz exponent.
\newblock {\em Proceedings of the American Mathematical Society},
  142(9):3089--3102, 2014.

\bibitem{LewPang}
Adrian Lewis and Jong~Shi Pang.
\newblock Error bounds for convex inequality systems.
\newblock In {\em Generalized convexity, generalized monotonicity: recent
  results}, pages 75--110. Springer, 1998.

\bibitem{Li10}
Guoyin Li.
\newblock On the asymptotically well behaved functions and global error bound
  for convex polynomials.
\newblock {\em SIAM Journal on Optimization}, 20(4):1923--1943, 2010.

\bibitem{Li}
Guoyin Li.
\newblock Global error bounds for piecewise convex polynomials.
\newblock {\em Mathematical Programming}, pages 1--28, 2013.

\bibitem{LMNP}
Guoyin Li, Boris Mordukhovich, TTA Nghia, and Tien~Son Pham.
\newblock Error bounds for parametric polynomial systems with applications to
  higher-order stability analysis and convergence rates.
\newblock {\em Mathematical Programming}, pages 1--34, 2015.

\bibitem{LiMorPham}
Guoyin Li, Boris Mordukhovich, and Tien~Son Pham.
\newblock New fractional error bounds for polynomial systems with applications
  to h{\"o}lderian stability in optimization and spectral theory of tensors.
\newblock {\em Mathematical Programming}, 153(2):333--362, 2015.

\bibitem{WLi95}
Wu~Li.
\newblock Error bounds for piecewise convex quadratic programs and
  applications.
\newblock {\em SIAM Journal on Control and Optimization}, 33(5):1510--1529,
  1995.

\bibitem{WLi97}
Wu~Li.
\newblock Abadie's constraint qualification, metric regularity, and error
  bounds for differentiable convex inequalities.
\newblock {\em SIAM Journal on Optimization}, 7(4):966--978, 1997.

\bibitem{Loja59}
Stanis{\l}aw {\L}ojasiewicz.
\newblock Sur le probleme de la division.
\newblock {\em Studia Mathematica}, 18, 1959.

\bibitem{Loja63}
Stanis{\l}aw {\L}ojasiewicz.
\newblock Une propri{\'e}t{\'e} topologique des sous-ensembles analytiques
  r{\'e}els.
\newblock {\em Les {\'e}quations aux d{\'e}riv{\'e}es partielles}, 117:87--89,
  1963.

\bibitem{LuoLuo}
Xiao~Dong Luo and Zhi~Quan Luo.
\newblock Extension of hoffman’s error bound to polynomial systems.
\newblock {\em SIAM Journal on Optimization}, 4(2):383--392, 1994.

\bibitem{LuoPang}
Zhi~Quan Luo and Jong~Shi Pang.
\newblock Error bounds for analytic systems and their applications.
\newblock {\em Mathematical Programming}, 67(1-3):1--28, 1994.

\bibitem{LuoSturm}
Zhi~Quan Luo and Jos Sturm.
\newblock Error bounds for quadratic systems.
\newblock In {\em High performance optimization}, pages 383--404. Springer,
  2000.

\bibitem{LuoTse922}
Zhi~Quan Luo and Paul Tseng.
\newblock Error bound and convergence analysis of matrix splitting algorithms
  for the affine variational inequality problem.
\newblock {\em SIAM Journal on Optimization}, 2(1):43--54, 1992.

\bibitem{LuoTseng93}
Zhi~Quan Luo and Paul Tseng.
\newblock On the linear convergence of descent methods for convex essentially
  smooth minimization.
\newblock {\em SIAM Journal on Control and Optimization}, 30(2):408--425, 1992.

\bibitem{LuoTseng}
Zhi~Quan Luo and Paul Tseng.
\newblock Error bounds and convergence analysis of feasible descent methods: a
  general approach.
\newblock {\em Annals of Operations Research}, 46(1):157--178, 1993.

\bibitem{Man}
Olvi Mangasarian.
\newblock A condition number for differentiable convex inequalities.
\newblock {\em Mathematics of Operations Research}, 10(2):175--179, 1985.

\bibitem{NgZheng}
Kung~Fu Ng and Xi~Yin Zheng.
\newblock Global error bounds with fractional exponents.
\newblock {\em Mathematical programming}, 88(2):357--370, 2000.

\bibitem{NgFeng1}
Kung~Fu Ng and Xi~Yin Zheng.
\newblock Error bounds for lower semicontinuous functions in normed spaces.
\newblock {\em SIAM Journal on Optimization}, 12(1):1--17, 2001.

\bibitem{Ngai}
Huynh~Van Ngai.
\newblock Global error bounds for systems of convex polynomials over polyhedral
  constraints.
\newblock {\em SIAM Journal on Optimization}, 25(1):521--539, 2015.

\bibitem{NgaiThera05}
Huynh~van Ngai and Michel Th{\'e}ra.
\newblock Error bounds for convex differentiable inequality systems in banach
  spaces.
\newblock {\em Mathematical programming}, 104(2):465--482, 2005.

\bibitem{Son10}
Tien~Son Pham.
\newblock The {\l}ojasiewicz exponent of a continuous subanalytic function at
  an isolated zero.
\newblock {\em Proceedings of the American Mathematical Society}, 139(1):1--9,
  2011.

\bibitem{Son12}
Tien~Son Pham.
\newblock An explicit bound for the {\l}ojasiewicz exponent of real
  polynomials.
\newblock {\em Kodai Mathematical Journal}, 35(2):311--319, 2012.

\bibitem{Rob}
Stephen Robinson.
\newblock An application of error bounds for convex programming in a linear
  space.
\newblock {\em SIAM Journal on Control}, 13(2):271--273, 1975.

\bibitem{Tse}
Paul Tseng.
\newblock On linear convergence of iterative methods for the variational
  inequality problem.
\newblock {\em Journal of Computational and Applied Mathematics},
  60(1-2):237--252, 1995.

\bibitem{TseYun}
Paul Tseng and Sangwoon Yun.
\newblock A coordinate gradient descent method for nonsmooth separable
  minimization.
\newblock {\em Mathematical Programming}, 117(1):387--423, 2009.

\bibitem{NgaiThera09}
Huynh Van~Ngai and Michel Th{\'e}ra.
\newblock Error bounds for systems of lower semicontinuous functions in asplund
  spaces.
\newblock {\em Mathematical Programming}, 116(1):397--427, 2009.

\bibitem{WangLin}
Po~Wei Wang and Chih~Jen Lin.
\newblock Iteration complexity of feasible descent methods for convex
  optimization.
\newblock {\em Journal of Machine Learning Research}, 15(1):1523--1548, 2014.

\bibitem{Pangwang}
Tao Wang and Jong~Shi Pang.
\newblock Global error bounds for convex quadratic inequality systems.
\newblock {\em Optimization}, 31(1):1--12, 1994.

\bibitem{WuYe02}
Zili Wu and Jane Ye.
\newblock On error bounds for lower semicontinuous functions.
\newblock {\em Mathematical programming}, 92(2):301--314, 2002.

\bibitem{WuYe01}
Zili Wu and Jane Ye.
\newblock Sufficient conditions for error bounds.
\newblock {\em SIAM Journal on Optimization}, 12(2):421--435, 2002.

\bibitem{Yang}
Wein~Hong Yang.
\newblock Error bounds for convex polynomials.
\newblock {\em SIAM Journal on Optimization}, 19(4):1633--1647, 2009.

\end{thebibliography}
\end{document}